\documentclass{amsart}
\usepackage{amsmath}
\usepackage{amssymb}
\usepackage[english]{babel}
\usepackage{graphics}
\usepackage[all]{xy}
\usepackage{amsrefs}
\usepackage{amsmath}
\usepackage{amsthm}
\usepackage{amssymb}
\usepackage{amsfonts}
\usepackage{amsxtra}     
\usepackage{tikz}
\usepackage{multicol}
\usepackage{subfigure}
\usepackage{centernot}
\usepackage{soul}
\usepackage{xcolor}
\usepackage{longtable}
\usepackage{layout}
\usepackage{geometry}
\usepackage{amsthm}
\usepackage{tikz, tikz-cd, quiver}

\usepackage{amstext} 
\usepackage{array}   
\newcolumntype{L}{>{$}c<{$}} 
\usepackage{longtable}

\usepackage[colorlinks=true,linkcolor=blue,urlcolor=blue, citecolor=blue]{hyperref}
\usepackage{fullpage}
\usepackage{epsfig}
\usepackage{verbatim}
\usepackage{enumitem}
\usepackage{algorithmicx}
\usepackage{float}
\usepackage{algpseudocode}

\usepackage[justification=centering]{caption}

\newtheorem{Theorem}{Theorem}[section]

\newtheorem{Corollary}[Theorem]{Corollary}
\newtheorem{Lemma}[Theorem]{Lemma}

\newtheorem{Conjecture}[Theorem]{Conjecture}

\theoremstyle{definition}

\newcommand{\PSL}{\mathrm{PSL}}
\newcommand{\Sz}{\mathrm{Sz}}

\newcommand{\SolG}[1]{\text{Sol}_G(#1)}
\newcommand{\sgraph}{\Gamma_S(G)}
\newcommand{\dgraph}{\Delta_S(G)}
\newcommand{\gen}[1]{\langle #1 \rangle}
\newcommand{\cc}{C_2^p \rtimes C_{q-1}}
\newcommand{\ccc}{C_3^p \rtimes C_{(q-1)/2}}
\newcommand{\cccc}{(C_2^p.C_2^p) \rtimes C_{q-1}}
\newcommand{\cp}{C_p\rtimes C_{(p-1)/2}}

\newcommand{\divides}{\bigm|}
\newcommand{\ndivides}{%
  \mathrel{\mkern.5mu 
    \ooalign{\hidewidth$\big|$\hidewidth\cr$\nmid$\cr}%
  }%
}

\theoremstyle{remark}

\usepackage{graphicx}
\usepackage{adjustbox}

\makeatletter
\newcommand*\bigcdot{\mathpalette\bigcdot@{.5}}
\newcommand*\bigcdot@[2]{\mathbin{\vcenter{\hbox{\scalebox{#2}{$\m@th#1\bullet$}}}}}
\makeatother

\DeclareMathOperator{\Core}{Core}

\title{CHaracterization of Solubilizers of Elements in Minimal Simple Groups}
\author[B. ~Akbari]{Banafsheh Akbari}
\author[J.~Chuharski]{Jake Chuharski}
\author[V. ~Sharan]{Vismay Sharan}
\author[Z. ~Slonim]{Zachary Slonim}
\address{Banafsheh~Akbari\\ Department of Mathematics \\ Cornell University \\ Ithaca, New York, USA \\
  \href{mailto:b.akbari@cornell.edu}
  {{\ttfamily\upshape b.akbari@cornell.edu}}}
  
\address{Jake~Chuharski\\ Department of Mathematics \\ Massachusetts Institute of Technology\\ Cambridge, Massachusetts, USA \\
  \href{mailto:chuharsk@mit.edu}
  {{\ttfamily\upshape chuharsk@mit.edu}}}

\address{Vismay~Sharan\\ Department of Mathematics \\ Yale University \\ New Haven, Connecticut, USA \\
  \href{mailto:vismay.sharan@yale.edu}
  {{\ttfamily\upshape vismay.sharan@yale.edu}}}
  
\address{Zachary~Slonim\\ Department of Mathematics \\ University of California, Berkeley \\ Berkeley, California, USA \\
  \href{mailto:zachslonim@berkeley.edu}
  {{\ttfamily\upshape zachslonim@berkeley.edu}}}
\begin{document}

\maketitle

\begin{abstract}
Given a finite group $G$, the solubilizer
of an element $x$, denoted by $\SolG{x}$, is the set of all elements $y$ such that $\langle x, y\rangle$ is a soluble subgroup of $G$. In this paper, we provide a classification for all solubilizers of elements in minimal simple groups. 
We also examine these sets to explore their properties by discussing some computational methods and making some conjectures for further work. \\[0.3cm]
{\bf Keywords}: Soluble group; Simple groups; Solubilizer, Solubility Graph.
\end{abstract}

\def\thefootnote{ \ }
\footnote{{\em $2000$ Mathematics Subject Classification}:
20D10, 20D08, 20D60.}

\begin{section}{Introduction}\label{Introduction}
Let $G$ be a finite group. For an element $x\in G$, we define the {\em solubilizer}
of $x$ in $G$ as the set
    \[\SolG{x}:=\{ y\in G \mid \gen{x,y} \text{ is soluble}\}.\]
In general, $\SolG{x}$ is not a subgroup of $G$. However, it can happen that this set is a subgroup. In fact, it has been shown in \cite{Akbari2} that $\SolG{x}$ is a subgroup of $G$ for any element $x\in G$ if and only if $G$ is a soluble group.
    
Due to a well known result, a finite group $G$ is soluble if and only if, for every $x, y\in G$ the subgroup $\langle x, y\rangle$ is
soluble (see Thompson \cite{Thompson} and Flavell \cite{Flavell}). This means that a ﬁnite group $G$ is soluble if and only if, for any $x\in G$, $\SolG{x}=G$.

We denote by $R(G)$ the soluble radical of $G$, which is the largest soluble normal subgroup of $G$. In \cite{GKPS}, Guralnick et al. proved that for an element $x$ in $G$, $x\in R(G)$ if and only if the subgroup $\gen{x,y}$ is soluble for all $y\in G$. Hence, $x\in R(G)$ if and only if $\SolG{x}=G$. In Section \ref{prelim} of this paper, we give some basic properties of the solubilizer sets, some interesting known results, and some useful lemmas that we use later.

As an interesting problem, we can consider how the properties of the structure of a single solubilizer can affect the structure of the whole group. For instance, it has been proved in \cite{Akbari2} that if $G$ is a group having an element $x$ such that all elements of $\SolG{x}$ commute pairwise, then $G$ is abelian. In \cite{Akbari3}, this was further generalized to show that if $G$ has an element $x$ such that $[u_1, u_2, u_3]=1$ for every $u_1, u_2, u_3\in \SolG{x}$, then $\gamma_3(G)=1$. That is, the subgroup generated by all long commutators of weight $3$ is the third term of the lower central series of $G$, implying that $G$ is a nilpotent group if all such commutators are trivial. We also note that the arithmetic properties of the solubilizers of elements can directly influence the structure of the group. For instance, due to the results obtained in \cite{Akbari2} and \cite{Akbari3}, we can see that if a group $G$ has an element whose solubilizer has size $p$ or $p^2$, where $p$ is a prime, then $G$ is a $p$-group.

In this paper, we continue exploring the solubilizer set to find more information about it. We seek to obtain more information about how exactly the properties of a single solubilizer set influence the entire group. To do this, we provide a list of all solubilizers and their sizes for the minimal simple groups as an infinite class of finite groups in Section \ref{class}. A  minimal simple group is a non-abelian simple group in which every proper subgroup is soluble. Thompson \cite{Thompson}*{Corollary 1} determined all minimal simple groups. In fact, every minimal simple group is isomorphic to one of the following: 
The projective special linear groups  $\PSL(2, 2^p)$, where $p$ is any prime; $\PSL(2,3^p)$ where $p$ is an odd prime; $\PSL(2, p)$ where $p > 3$ is a prime so that $5 \mid p^2 + 1$ or equivalently so that $p\equiv 2$ or $3\pmod 5$; $\PSL(3,3)$; and the Suzuki groups $\Sz(2^p)$, where $p$ is an odd prime. We provide some tables including some information about the structure of subgroups containing a certain element in minimal simple groups which allows us to find the solubilizer sets and their sizes.

In Section \ref{graphprop}, we investigate some properties of the induced solubility graph $\Delta_S(G)$, a simple graph defined on the finite group $G$ whose vertices are the elements of $G\setminus R(G)$ and two vertices $x,y\in G\setminus R(G)$ are adjacent if and only if $\gen{x,y}$ is soluble. We prove some conditions for when $\Delta_S(G)$ has an Eulerian or Hamiltonian cycle, and look specially at when $G$ is a minimal simple group. We also provide some bounds on the chromatic number of $\Delta_S(G)$ for a minimal simple group $G$ in Sections \ref{graphprop} and \ref{cors}. Then, we give some results using our classification to answer some interesting conjectures about the solubilizer sets for minimal simple groups. Finally, in Section \ref{comp}, we give some computational results we obtained regarding colorability and Hamiltonian cycles using algorithms we ran on the adjacency matrices of $\Delta_S(G)$ for some small simple groups.
\end{section}

\begin{section}{Preliminary Results}\label{prelim}
The solubilizer of an element in a finite group was introduced in \cite{HaiReuven}. This set has been studied extensively in \cite{Akbari2, Akbari3, Mousavi}. This section is intended to collect some recent results on the solubilizer of an element and explore some additional properties of this set. Finally, we present some results which help us prove the results in Section \ref{class}.
\begin{Lemma}[\cite{Akbari2}*{Lemma 2.1, Corollary 2.2, and Lemma 2.4}]\label{normalizer} Let $G$ be a group and $x \in G$. Then:
\begin{itemize}
    \item[$(a)$] $\langle x \rangle\subsetneq N_G(\langle x \rangle) \subseteq N_G(\langle x \rangle) \cup R(G) \subseteq \SolG{x}$;
    
    \item[$(b)$] $|\SolG{x}|$ is divisible by $|x|$ and $|R(G)|$.
\end{itemize}
\end{Lemma}
In view of part $(a)$ of Lemma \ref{normalizer}, it is easy to check that given an element $x\in G$, if
$N_G(\langle x \rangle)$ is a maximal subgroup of $G$ which is the only one containing $x$, then $\SolG{x}=N_G(\langle x \rangle)$. 

\begin{Lemma}[\cite{Akbari2}*{Lemma 2.1}]\label{solvable def}
Let $G$ be a finite group and $x\in G$, then  $$\SolG{x}=\bigcup_{x\in H\le G}H$$
with the union being taken over all soluble subgroups of $G$ containing $x$.
\end{Lemma}

Given a normal soluble subgroup $N$ of $G$, we define
$$\frac{\SolG{x}}{N} = \{yN \mid  y \in \SolG{x}\} = \{yN \mid \langle x,y \rangle \hbox{ is soluble} \}.$$
We can apply the following result to reduce most of the conditions to the case when the soluble radical is trivial.
\begin{Lemma}[\cite{Akbari2}*{Lemmas 2.4 and 2.5}]\label{quotient}
Let $N$ be a normal soluble subgroup of a group $G$. Then $|\SolG{x}|$ is divisible by $|N|$, and $\SolG{G/N}(xN)=\SolG{x}/N$. In particular, $$\bigg|\frac{\SolG{x}}{N}\bigg|=\frac{|\SolG{x}|}{|N|}.$$
\end{Lemma}

As mentioned before, generally the solubilizer of an element of a group is not a subgroup.  So it would be reasonable to ask what can occur if this set is a subgroup. In fact, we have the following lemma.
\begin{Lemma}[\cite{Akbari2}*{Theorem 3.4}]\label{all solubilizer subgroup}
A group $G$ is soluble if and only if $\SolG{x}$ is a subgroup of $G$ for all $x\in G$.
\end{Lemma}
However, given an insoluble group $G$, we can often find some element $x$ such that $\SolG{x}$ is a subgroup. Then the normal core of $\SolG{x}$, that is the largest normal subgroup of $G$ being contained in $\SolG{x}$, is a subgroup of the soluble radical $R(G)$. So we have the following lemma.
\begin{Lemma}\label{core of sol}
Let $G$ be a group and $x$ an element of $G$ such that $\SolG{x}$ is a subgroup of $G$. Then $\Core_G(\SolG{x})$ is soluble.
\end{Lemma}
\begin{proof}
We will show that the normal core of $\SolG{x}$ is the union of the normal core of soluble subgroups of $G$ containing $x$. 

Using Lemma \ref{solvable def}, we have 
\[ \Core_G(\SolG{x})=\bigcap_{g\in G}(\bigcup_{x\in H\le G}H)^g=\bigcap_{g\in G}(\bigcup_{x\in H\le G}H^g)= \bigcup_{x\in H\le G}(\bigcap_{g\in G}H^g)=\bigcup_{x\in H\le G}\Core_G(H). \]

If the soluble radical $R(G)$ is trivial, then we can see that for any soluble subgroup of $G$ containing $x$, $\Core_G(H)$ is trivial and consequently, $\Core_G(\SolG{x})$ is trivial too. So assume that $R(G)$ is nontrivial. Then the soluble radical of $G/R(G)$ is trivial. Considering Lemma \ref{quotient}, we see that $\SolG{G/R(G)}(xR(G))$ is a subgroup. So we can use the last case to find $\Core_{G/R(G)}(Sol_{G/R(G)}(xR(G)))$ is trivial. On the other hand,
$$\Core_{G/R(G)}(Sol_{G/R(G)}(xR(G)))=\frac{\Core_G(\SolG{x})R(G)}{R(G)}.$$
This implies that $\Core_G(\SolG{x})\le R(G)$ and so $\Core_G(\SolG{x})$ is a soluble subgroup of $G$.
\end{proof}

When $\SolG{x}$ is a subgroup of an insoluble group $G$, we can also consider the case where $\SolG{x}$ is soluble or nilpotent.
\begin{Lemma}\label{soluble sol}
Let $G$ be an insoluble group and $x$ an element of $G$ such that $\SolG{x}$ is a soluble subgroup of $G$. Then $\SolG{x}$ is a maximal soluble subgroup of $G$.
\end{Lemma}
\begin{proof}
Note that if there exists a soluble subgroup $H$ of $G$ such that $\SolG{x}\le H<G$, then by Lemma \ref{solvable def},
 we will have $H\le \SolG{x}$ and so $\SolG{x}=H$.
\end{proof}

By a similar argument to Lemma \ref{soluble sol}, we can prove the following lemma.
\begin{Lemma}\label{nilpotent sol}
Let $G$ be an insoluble group and $x$ an element of $G$ such that $\SolG{x}$ is a nilpotent subgroup of $G$. Then $\SolG{x}$ is a maximal nilpotent subgroup of $G$.
\end{Lemma}
As a special case of Lemma \ref{nilpotent sol}, we can consider the case when  $\SolG{x}$ is a Sylow $p$-subgroup of $G$ for some element $x$ in $G$ and an odd prime $p$. Then it follows that $G$ is a $p$-group. More precisely, we have the following lemma.
\begin{Lemma}\cite{Akbari3}*{Lemma 4.1}\label{power of prime}
Let $G$ be an insoluble group and let $x$ be an element of $G$ such that $\SolG{x}$ is a subgroup. Then $|\SolG{x}|\neq p^n$ for all odd primes $p$ and all positive integers $n$.
\end{Lemma}
Lemma \ref{power of prime} is not always true for the case of $p=2$. Let $G=\PSL(2, 31)$ and $x$ an element of order $16$. We can see from Atlas of Finite Groups \cite{Atlas} that the maximal subgroup isomorphic to the dihedral group $D_{32}$ is the only soluble subgroup containing $x$. It follows that $\SolG{x}\cong D_{32}$, which is a Sylow $2$-subgroup. To find more information about this case see \cite{Mousavi}.

In \cite{Akbari2} and \cite{Akbari3}, the authors answered a slightly more general question about when $|\SolG{x}|$ can equal $p^n$ even if $\SolG{x}$ is not necessarily a subgroup of $G$. Note that the following result is true for any prime. 
\begin{Lemma}\cite{Akbari3}*{Theorem B}\label{primepower}
Let $G$ be an insoluble group and $x$ an element of $G$. Then the size of $\SolG{x}$ cannot be $p$ nor $p^2$ for any prime $p$.
\end{Lemma}
In Section \ref{cors}, we will generalize Lemma \ref{primepower} in the case where $p$ is an odd prime in all minimal simple groups.

In the sequel, we provide an applicable way to find the solubilizers of elements in the minimal simple groups which will be our goal in Section \ref{class}.

Clearly, to use Lemma \ref{solvable def}, we can find $\SolG{x}$ just as the union of all maximal soluble subgroups containing $x$.
\begin{Lemma}\label{sol in minimal simple}
Let $G$ be a minimal simple group and $x$ an element of $G$. Then $\SolG{x}=\bigcup_{x\in M< G}M$, where the union ranges over the maximal subgroups of $G$ containing $x$.
\end{Lemma}
\begin{proof}
Since all maximal subgroups of a minimal simple group are soluble, it suffices to simply take the union of all maximal subgroups of $G$ containing $x$.
\end{proof}
To compute all maximal subgroups containing a certain element in a minimal simple group in Section \ref{class}, we will use the following lemmas and the fact that given a finite group $G$ and a subgroup $H\le G$, the number of conjugates of $H$ in $G$ is $[G:N_G(H)]$. Assuming that $G$ is a simple group and $M\le G$ is a maximal subgroup of $G$, there are $[G:M]$ conjugate subgroups of $M$ is $G$ because $N_G(M)=M$.
\begin{Lemma}[\cite{HaiReuven}*{Lemma 2.11}]\label{solcong}
    Let $G$ be a finite group and suppose $x,y\in G$ are conjugate elements such that $x=g^{-1}yg.$ Then, $\SolG{x}=g^{-1}\SolG{y}g.$
\end{Lemma}
Note that the following result can be also true for any subgroup of a finite group $G$.
\begin{Lemma}\label{formula}
    Let $M$ be a maximal subgroup of a finite group $G$ and $x$ an element of $M$ with $m$ conjugates in $M$ and $n$ conjugates in $G$. 
    Suppose that the number of subgroups of $G$ conjugate to $M$ is $r$. Then the number of subgroups conjugate to $M$ containing $x$ is equal to $mr/n$.
\end{Lemma}
\begin{proof}
    Since we are counting over conjugates of $x$, the number of subgroups conjugate to $M$ that each conjugate of $x$ is contained in must be equal, say to some constant $k$. Then we can count the number of conjugates of $x$ in $G$ in two ways. By assumption we know that this number is $n$. We can also count this by counting over individual subgroups $M\leq G$. So we would get that the number of conjugates of $x$ in $M$ would equal the number of conjugates of $M$ in $G$ multiplied by the number of conjugates of $x$ in $M$ and then divided by $k$ since this would count every conjugate of $x$ exactly $k$ times. This shows that the number of conjugates of $x$ in $G$ would be $n=mr/k$ so $k=mr/n$.
\end{proof}
\begin{Lemma}\label{formula2}
    Let $M$ be a maximal subgroup of a finite group $G$ and suppose $x$ is an element of $G$ contained in a unique subgroup isomorphic to $C_s$ for some $s> 4$. Assume that the number of conjugate subgroups of $C_s$ in $M$ is $m$ and the number of conjugate subgroups of $M$ in $G$ is $r$. Let $n$ be the number of conjugate subgroups of $C_s$ in $G$. Then the number of subgroups conjugate to $M$ containing $x$ is equal to $mr/n$.
\end{Lemma}
\begin{proof}
    This follows in the same way as Lemma \ref{formula} except we count conjugates of the entire subgroup $C_s$ instead of a single element.
\end{proof}

\end{section}

\begin{section}{Classification Theorems}\label{class}

In this section, we classify $\SolG{x}$ for any element $x$ in a minimal simple group $G$. According to Lemma \ref{sol in minimal simple}, this set is the union of maximal subgroups containing $x$. More precisely, to find $\SolG{x}$, we need to count the number of maximal subgroups containing $x$ and give the pairwise intersections. Thus, for each of the five types of minimal simple groups \cite{Thompson}*{Corollary 1}, we produce a table which classifies the solubilizers.

\begin{Theorem}\label{2p}
    Let $G=\PSL(2, 2^p)$ where $p$ is a prime and let $q:=2^p$. Then the following table classifies all subsets $\SolG{x}$ for $x\in G$ based on the order of $x$ in terms of the number of maximal subgroups of each isomorphism type contained in $\SolG{x}$, their pairwise intersections and $|\SolG{x}|$:
\end{Theorem}
\begin{table}[H]
    \centering
    \begin{tabular}{| c || c | c | c|} 
     \hline
     \quad Maximal Subgroup & $|x| = 2$ & $|x| \divides q-1$ & $|x| \divides q+1$  \\ [0.5ex] 
     \hline\hline
     $\cc$ & 1 & 2 & 0  \\ 
     \hline
     $D_{2(q-1)}$ & $q/2$& 1 & 0  \\
     \hline
     $D_{2(q+1)}$ & $q/2$ & 0 & 1   \\
     \hline \hline
     Intersections & $\cong C_2$ & $\cong C_{q-1}$ & N/A \\
     \hline
     $|\SolG{x}|$ & $3q(q-1)$ & $2q(q-1)$ & $2(q+1)$ \\
     \hline
    \end{tabular}
    \caption{Classification of $\SolG{x}$ for $G=\PSL(2, 2^p)$}
    \label{tbl2,2p}
\end{table}

\begin{proof}
    We obtain the classification of the maximal subgroups of $G$ from \cite{King}*{Corollary 2.2}, which also gives that there is only one conjugacy class of each. Since $|G|=q(q-1)(q+1)$, an element $x\in G$ has order $|x|$ dividing $q,q-1,$ or $q+1.$ Since $q-1$ and $q+1$ are both odd, the only even order elements in $D_{2(q+1)}$ and $D_{2(q-1)}$ are involutions. Similarly, the elements of $\cc$ have order either $2$ or some divisor of $q-1.$ Hence, we know that the only even order elements of $G$ are involutions. Further, note that $\gcd(q-1,q+1)=1$ since $q\pm 1$ is odd. Hence, if $x$ is not the identity, then $|x|$ is in exactly one of the cases $|x|=2,|x|\divides q-1$ or $|x|\divides q+1.$ 
    
    We can see from \cite{King}*{Theorem 2.1} that there are $q+1$ conjugate abelian Sylow $2$-subgroups in $G$ which contain all the involutions of $G$.  Since any of these Sylow $2$-subgroups has $q-1$ involutions, the total number of involutions in $G$ is $q^2-1$.  All of these involutions are conjugate so their solubilizers are conjugate by Lemma \ref{solcong}. Further, from the same source, we know that there are $q(q+1)/2$ subgroups of $G$ isomorphic to $C_{q-1}$ which are all conjugate, and similarly there are $q(q-1)/2$ conjugate subgroups isomorphic to $C_{q+1}$. Now, we claim that if $|x|\divides q\pm 1$, then $x$ is contained in a unique subgroup isomorphic to $C_{q\pm 1}.$ 
    
    This is because, if $|x|\divides q\pm 1$, then $|\gen{x}|$ is odd and we can see from the subgroups of the maximal subgroups of $G$ that the only odd order subgroups are contained in subgroups isomorphic to $C_{q\pm 1}.$ Then, let $y\in G$ generate the $C_{q\pm 1}$ that contains $x.$ So we can write $x=y^k$ for some $k\in \mathbb{N}.$ Now, it is easy to see that $N_G(\gen{y}) \le N_G(\gen{x})$.

     Each maximal subgroup isomorphic to $D_{2(q\pm 1)}$ contains exactly one $C_{q\pm 1}$ subgroup. Since all the $C_{q\pm 1}$ subgroups are conjugate, they must each be contained in some $D_{2(q\pm 1)}$ because at least one $C_{q\pm 1}$ is contained in a $D_{2(q\pm 1)}$ and they all behave in the same way. Since $C_{q\pm1}$ is normal in any $D_{2(q\pm 1)}$ and $D_{2(q\pm 1)}$ is maximal in $G$, we can get that $N_G(C_{q\pm 1})\cong D_{2(q\pm 1)}$ because $G$ is a simple group. Thus each $C_{q\pm 1}$ is contained in exactly one $D_{2(q\pm 1)}$ because otherwise $D_{2(q\pm 1)}\lneq N_G(C_{q\pm 1}).$

    Hence, $D_{2(q\pm 1)} = N_G(\gen{y}) \le N_G(\gen{x})$ and so we must have equality since $D_{2(q\pm 1)}$ is maximal in $G$. It follows that $x$ is contained in exactly one subgroup isomorphic to $D_{2(q\pm 1)}$ which means it is contained in exactly one subgroup isomorphic to $C_{q\pm 1}$, as claimed.

    Now, we calculate the number of maximal subgroups of each isomorphism type that contain $x$ using the fact that each isomorphism type consists of one conjugacy class of subgroups and the number of conjugates of the maximal subgroup $M$ in $G$ is $[G:M]$. There are $q+1$ conjugates of $\cc$, $q(q-1)/2$ conjugates of $D_{2(q+1)}$, and $q(q+1)/2$ conjugates of $D_{2(q-1)}$. Finally, we use Lemmas \ref{formula} and \ref{formula2} to calculate the number of maximal subgroups of each type containing $x$.
    
    For the element $x$ that is an involution, there are $q^2-1$ conjugates of $x$ in $G,$ and there are $q-1$, $q+1$, and $q-1$ conjugates of $x$ in $\cc, D_{2(q+1)}$, and $D_{2(q-1)}$ respectively. Similarly there are $q+1$, $q(q-1)/2$, and $q(q+1)/2$ subgroups of $G$ isomorphic to $\cc, D_{2(q+1)},$ and $D_{2(q-1)}$ respectively. So $x$ is contained in $\frac{(q-1)(q+1)}{q^2-1}=1$ subgroup isomorphic to $\cc, \frac{(q+1)(q(q-1)/2)}{q^2-1}=q/2$ subgroups isomorphic to $D_{2(q+1)}$, and $\frac{(q-1)(q(q+1)/2)}{q^2-1}=q/2$ subgroups isomorphic to $D_{2(q-1)}.$

    For $|x|\divides q\pm 1,$ we already have that $x$ is contained in exactly one subgroup isomorphic to $D_{2(q\pm 1)}.$ Then, if $|x|\divides q-1,$ it is also contained in $\frac{(q+1)q}{q(q+1)/2}=2$ subgroups isomorphic to $\cc.$ Hence, we get Table \ref{tbl2,2p} produced at the beginning of the proof.

    Next, we turn our attention to the intersections between these maximal subgroups that contain $x.$ In the case of the involutions, the intersection at least contains $C_2$. Note that $q\pm 1$ is odd and so $C_2$ is maximal in $D_{2(q\pm 1)}$ so the intersections must all be exactly isomorphic to $C_2.$ Similarly, for $|x|\divides q-1,$ the intersections must contain the unique $C_{q-1}$ that contains $x$. Clearly, $C_{q- 1}$ is maximal in $D_{2(q-1)}$ so again the intersections must all be isomorphic to $C_{q-1}.$ Finally, we calculate $|\SolG{x}|$ for each column using the principle of inclusion exclusion (PIE) as follows:

    \[\Bigg{|}\bigcup_{i=1}^m M_i\Bigg{|} = \sum_{i=1}^m |M_i| - \sum_{1\le i<j\le m}|M_i\cap M_j| + \sum_{1\le i<j<k\le m}|M_i\cap M_j\cap M_k| - \ldots  + (-1)^{m-1}|M_1\cap M_1\cap \cdots M_m|,\]
    where $M_1, \cdots, M_m$ are the maximal subgroups containing $x$.
    
    For $|x|\divides q+1$, there are no intersections so $|\SolG{x}|=|D_{2(q+1)}|=2(q+1)$. For $|x|\divides q-1$, the pairwise intersection between any of the three maximal subgroups containing $x$ has order $|C_{q-1}|=q-1$ and the three-way intersection between all three also has order $q-1$. So $$|\SolG{x}|=2\cdot |\cc|+1\cdot |D_{2(q-1)}|-3\cdot |C_{q-1}|+1\cdot |C_{q-1}|$$ $$\ \ \ \ \ \ \ \ \ \ =2q(q-1)+2(q-1)-3(q-1)+(q-1)=2q(q-1).$$ Finally, let $x$ be an involution. First note that there are $m=1+q/2+q/2=q+1$ maximal subgroups containing $x$ and if we choose any $i$ different subgroups of these maximal subgroups, the intersection has order $2$. Then we have:
     \[|\SolG{x}|=\Bigg{|}\bigcup_{i=1}^{q+1}
     M_i\Bigg{|}=1\cdot q(q-1)+q/2\cdot 2(q-1)+q/2\cdot 2(q+1)-2\binom{q+1}{2}+2\binom{q+1}{3}
    \]
    \[-\cdots+(-1)^q\cdot 2\binom{q+1}{q+1}=3q^2-q-2\Bigg{(}\sum_{i=2}^{q+1}(-1)^i\binom{q+1}{i} \Bigg{)}
    \]
    \[=3q^2-q-2\Bigg{(}(1-1)^{q+1}-1-(-1)q+1\Bigg{)}=3q^2-q-2q=3q(q-1).\]
    
    This completes the classification of the table for $G=\PSL(2, 2^p).$
\end{proof}

\begin{Theorem}\label{3p}
    Let $G=\PSL(2,3^p)$ where $p$ is an odd prime and denote $q:=3^p$. Then the following table classifies all subsets $\SolG{x}$ for $x\in G$ based on the order of $x$ in terms of the number of maximal subgroups of each isomorphism type contained in $\SolG{x}$, their pairwise intersections and $|\SolG{x}|$:
\end{Theorem}
\begin{table}[H]
    \centering
\begin{tabular}{| c || c | c | c | c|} 
 \hline
 \quad Maximal Subgroup &  $|x| = 2$ & $|x| = 3$ & $|x| \divides q-1$ & $|x| \divides q+1$  \\ [0.5ex] 
 \hline\hline
 $\ccc$ & 0 & 1 & 2 & 0  \\ 
 \hline
 $D_{q-1}$ & $(q+1)/2$& 0 & 1 & 0  \\
 \hline
 $D_{q+1}$ & $(q+3)/2$ & 0 & 0 & 1  \\
 \hline
 $\rm{A}_4$ & $(q+1)/4$ & $q/3$ & 0 & 0 \\
 \hline \hline
 Intersections & $\cong C_2$ or $C_2^2$ & $\cong C_3$ & $\cong C_{(q-1)/2}$ & N/A \\
 \hline
 $|\SolG{x}|$ & $q(q+1)$ & $q(q+5)/2$ & $q(q-1)$ & $q+1$ \\
 \hline
\end{tabular}
    \caption{Classification of $\SolG{x}$ for $G=\PSL(2,3^p)$}
    \label{tbl2,3p}
\end{table}

\begin{proof}
    First, we note that the last two columns of Table \ref{tbl2,3p} are the exact same as those in Table \ref{tbl2,2p} but with $q\pm 1$ replaced by $\frac{q\pm 1}{2}$ and the proof follows in the same way as in Theorem \ref{2p}. It is important to highlight that when writing  $|x|\divides q\pm 1$, we mean $|x|\divides q\pm 1$ but $|x|\neq2$ which is treated in the first column. For the first two columns of Table \ref{tbl2,3p}, we apply Lemmas \ref{formula} and \ref{formula2} and the fact that the number of conjugates of maximal subgroup $M$ in $G$ is $[G:M]$, in the same way as in the proof of Theorem \ref{2p}. We also use the fact that, as before, all the order $2$ elements are conjugate, all the $C_3$ subgroups are conjugate, and similarly \cite{King}*{Corollary 2.2} for the number of maximal subgroups of each isomorphism type. Note that each involution is contained in $(q+3)/2$ subgroups isomorphic to $D_{q+1}$ because $D_{q+1}$ has $(q+3)/2$ involutions in $2$ conjugacy classes, namely, $(q+1)/2$ reflections and $1$ rotation.
    
    If $|x|=3,$ the intersections must all contain $C_3$ and since $C_3$ is maximal in $\rm{A}_4,$ this is precisely what all the intersections must be. For $|x|=2,$ the intersections must all contain $C_2$ but $C_2$ is not maximal in $D_{q+1}$ and $\rm{A}_4$ so more care must be taken. We proceed in cases.

    As $q=3^p$ and $p$ is odd, we see that $q\equiv 3 \pmod 4$, so $q-1 \not \equiv 0 \pmod 4$. In particular, this means that $2\ndivides \frac{q-1}{2}$ so $C_2$ is maximal in $D_{q-1}$ because all subgroups of $D_{2n}$ are of the form $C_d$ or $D_{2d}$ for $d$ as a divisor of $n$. Thus, the intersection of $D_{q-1}$ and any other maximal subgroup containing $x$ must be isomorphic to $C_2.$

    For the intersection between two subgroups isomorphic to $\rm{A}_4,$ we note that $C_2$ is not maximal in $\rm{A}_4$ and is instead contained in a maximal $C_2^2$ subgroup. However, $C_2^2\lhd \rm{A}_4$ so $N_G(C_2^2)\cong \rm{A}_4$ since $\rm{A}_4$ is maximal in $G$ which is simple. If $M_1$ and $M_2$ are two maximal subgroups isomorphic to $\rm{A}_4$ such that  
    $M_1\cap M_2\cong C_2^2$, then we would have $M_1=N_G(M_1\cap M_2)=M_2$ contradicting the fact that $M_1$ and $M_2$ were assumed to be distinct subgroups containing $x.$ Thus, we can conclude that the intersection of any two $\rm{A}_4$ subgroups containing $x$ is also just $C_2.$

    Finally, we consider the intersections between two copies of $D_{q+1}$, and between $D_{q+1}$ and $\rm{A}_4$. Now, $q+1\equiv 0 \pmod 4$ and so each subgroup $D_{q+1}$ contains a single involution which is in a different conjugacy class to the rest of the involutions, namely a rotation of order $2$ rather than a reflection. Note that involutions are conjugate in $G$ and we already showed that each involution is contained in $(q+3)/2$ subgroups isomorphic to $D_{q+1}$. Thus, we see that each involution is contained in exactly one $D_{q+1}$ in which it is the singleton  conjugacy class of a rotational involution and $(q+1)/2$ copies of $D_{q+1}$ where it is in the conjugacy class of $q+1$ reflections. We call this $D_{q+1}$ the special $D_{q+1}$ corresponding to $x$, and denote it $D_{q+1}^s$ for ease of reference.
    
    From the properties of the dihedral group we see that $x$ is contained in $\frac{q+1}{4}$ different copies of $C_2^2$ in $D_{q+1}^s$. We also know that in each copy of $\rm{A}_4$ containing $x$ we have exactly one $C_2^2$ by the structure of $\rm{A}_4$. Further, every $C_2^2$ in $G$ must be contained in some $\rm{A}_4$ since all order $2$ elements are conjugate so all copies of $C_2^2$ are conjugate so behave in the same way. This implies that the $C_2^2$ subgroups are in a one-to-one correspondence with the $\rm{A}_4$ subgroups, with each $\rm{A}_4$ containing a distinct $C_2^2$. Then, since $D_{q+1}^s$ contains $\frac{q+1}{4}$ copies of $C_2^2$ containing $x$, this special dihedral group contains every $C_2^2$ containing $x$.
    
    This shows that the intersection of the special $D_{q+1}$ with any of the copies of $\rm{A}_4$ must be $C_2^2$. For every non-special $D_{q+1}$, we see that $x$ is not in a singleton conjugacy class, so it is contained in exactly one $C_2^2$, namely the one generated by $x$ and the rotational involution in that $D_{q+1}$. This implies that the intersection between $D_{q+1}^s$ and any of the other $D_{q+1}$s containing $x$ must be $C_2^2$. The reason that it cannot be in a larger dihedral group, is because if there was such a subgroup $D_{2d}\le D_{q+1}$ where $d\divides q+1$, then there would be an element $y\in D_{2d}$ of order $d>2$ which is contained in two copies of $D_{q+1}$. We showed this was impossible when considering the last column of the table. 
    
     Now we claim that for every $\rm{A}_4$ there are exactly two of the non-special $D_{q+1}$s with intersection $C_2^2$. Clearly, every non-special $D_{q+1}$ has intersection equal to $C_2^2$ with exactly one $\rm{A}_4$ since each $C_2^2$ corresponds to exactly one $\rm{A}_4$ and each non-special $D_{q+1}$ contains one $C_2^2$ subgroup containing $x.$
    
    In a non-special $D_{q+1}$, any subgroup isomorphic to $C_2^2$ that contains $x$ is of the form $\{1,x,s,xs\}$ where $s$ is the rotational involution of that $D_{q+1}$. This implies that if a given $\rm{A}_4$ has intersection larger than $\gen{x}=\{1,x\}$ with more than $2$ of these non-special $D_{q+1}$, then these two $D_{q+1}$s have the same rotational involution which is not possible since all of the involutions are conjugate and each is the rotational involution in exactly one $D_{q+1}$. This means that each $\rm{A}_4$ shares a $C_2^2$ with the special $D_{q+1}$ and two of the non-special $D_{q+1}$s. Note that the two non-special $D_{q+1}$ that share the $C_2^2$ with the same $\rm{A}_4$ will also share a $C_2^2$. This organizes these subgroups as shown in the figure with each edge representing a shared $C_2^2$ where there are $(q+1)/4$ copies of $\rm{A}_4.$

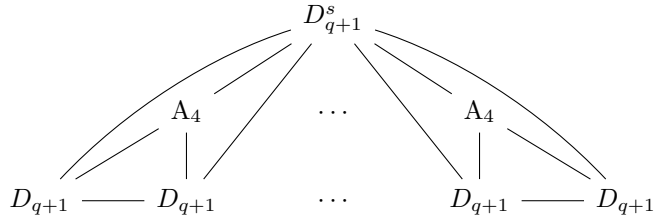
\begin{figure}[H]
    \centering
    \begin{tikzcd}
	&& {D_{q+1}^s} \\
	& {\rm{A}_4} & \cdots & {\rm{A}_4} \\
	{D_{q+1}} & {D_{q+1}} & \cdots & {D_{q+1}} & {D_{q+1}}
	\arrow[no head, from=2-2, to=3-1]
	\arrow[no head, from=2-2, to=3-2]
	\arrow[no head, from=2-4, to=3-4]
	\arrow[no head, from=2-4, to=3-5]
	\arrow[no head, from=3-1, to=3-2]
	\arrow[no head, from=3-4, to=3-5]
	\arrow[no head, from=1-3, to=2-2]
	\arrow[no head, from=1-3, to=2-4]
	\arrow[curve={height=12pt}, no head, from=1-3, to=3-1]
	\arrow[curve={height=-12pt}, no head, from=1-3, to=3-5]
	\arrow[no head, from=1-3, to=3-2]
	\arrow[no head, from=1-3, to=3-4]
\end{tikzcd}
    \caption{$C_2^2$ intersections of maximal subgroups containing an involution in $\PSL(2,3^p)$}
    \label{fig1}
\end{figure}

    As in the proof of Theorem \ref{2p}, we now use the principle of inclusion exclusion (PIE) to calculate $|\SolG{x}|.$ For the last three columns, this follows in the exact same way so we omit the calculations. However, for $|x|=2,$ we have a bit more work to do. First, we treat each intersection as having order $2$ and perform the same calculation as in Theorem \ref{2p}. 
    Then we account for the intersections which are $C_2^2$ having order $4$ by subtracting off an extra $2$ times the alternating sum of the number of such intersections. From Figure \ref{fig1}, we see that there are $6\cdot (q+1)/4$ pairwise intersections isomorphic to $C_2^2,$ there are $4\cdot (q+1)/4$ three-way intersections isomorphic to $C_2^2,$ and $1\cdot (q+1)/4$ four-way intersections isomorphic to $C_2^2$. 
     This gives us
    
     \[|\SolG{x}|=\frac{q+1}{2}\cdot(q-1)+\frac{q+3}{2}\cdot (q+1)+\frac{q+1}{4}\cdot 12 + 2\left(1-\left(\frac{q+1}{2}+\frac{q+3}{2}+\frac{q+1}{4}\right)\right)\] 
    \[+2\cdot (-6+4-1)\cdot \frac{q+1}{4}
    =q^2 + \frac{5q}{2} + \frac{3}{2}-3\cdot \frac{q+1}{2}=q^2+q.\ \ \ \ \ \ \ \ \ \ \ \ \  \]
    This completes the classification of the table for $\PSL(2,3^p).$
\end{proof}
Before considering the next case of the minimal simple groups, we state the following lemma which we will apply to find the size of solubilizer sets for involutions in a different way from inclusion exclusion principle because in this case the intersections between the maximal subgroups containing involutions are much more complicated to compute.
\begin{Lemma}\label{invol}
    Let $G$ be a finite group and $I(S)$ be the set of involutions of $S\subseteq G$. If all of the involutions in $G$ are conjugate and $x\in G$ is an involution, then
    \[|\SolG{x}|\cdot|I(G)|=\sum_{y\in G}|I(\SolG{y})|.\]
\end{Lemma}
\begin{proof}
Since all the involutions in $G$ are conjugate, any involution $z\in G$ is of the form $z=x^g$ for some $g\in G$. Now we can see from Lemma \ref{solcong} that $|\SolG{z}|=|\SolG{x}|$.  Thus $|\SolG{x}|\cdot|I(G)|=\sum_{z\in I(G)}|\SolG{z}|$. To compute $\sum_{z\in I(G)}|\SolG{z}|$, let $z$ be an involution of $G$. Given an element $y\in G$, if $y\in \SolG{z}$, then $z\in \SolG{y}$. Thus $z$ is an involution of $\SolG{y}$ and so $z\in I(\SolG{y})$. This completes the proof.
\end{proof}

\begin{Theorem}\label{p}
    Let $G=\PSL(2, p)$ where $p$ is a prime number such that $p \equiv 2$ or $3 \pmod 5$. Then the following tables classify all subsets $\SolG{x}$ for $x\in G$ based on the order of $x$ in terms of the number of maximal subgroups of each isomorphism type contained in $\SolG{x}$ and $|\SolG{x}|.$
\end{Theorem}

\begin{center}

\begin{table}[H]
\begin{tabular}{| c || c | c | c | c | c | c |} 
 \hline
 \quad Maximal Subgroup & $|x| = 2$ & $|x| = 3$ & $|x| = 4 $ & $|x| = p$ & $|x| \divides p-1$ & $|x|\divides p+1$  \\ [0.5ex] 
 \hline\hline
 $\cp$ & 2 & 2 & 2 & 1 & 2 & 0 \\ 
 \hline
 $D_{p-1}$ & $(p+1)/2$ & 1 & 1 & 0 & 1 & 0 \\
 \hline
 $D_{p+1}$ & $(p-1)/2$ & 0 & 0 & 0 & 0 & 1 \\
 \hline
 $\rm{S}_4$ & $3(p-1)/4$ & $(p-1)/3$ & $(p-1)/4$ & 0 & 0 & 0 \\
 \hline \hline
 $|\SolG{x}|$ & $(p-1)(2p+3)$ & $(p-1)(p+6)$ & $(p-1)(p+4)$ & $p(p-1)/2$ & $p(p-1)$& $p+1$\\
 \hline
\end{tabular}

\caption{Classification of $\SolG{x}$ for $G=\PSL(2, p)$ when $p\equiv 1\pmod{24}$}
    \label{tbl2,p,1}
\end{table}

\begin{table}[H]

\begin{tabular}{| c || c | c | c | c | c |} 
 \hline
 \quad Maximal Subgroup & $|x| = 2$ & $|x| = 3$ & $|x| = p$ & $|x| \divides p-1$ & $|x|\divides p+1$  \\ [0.5ex] 
 \hline\hline
 $\cp$ & 2 & 0 & 1 & 2 & 0 \\ 
 \hline
 $D_{p-1}$ & $(p+1)/2$ & 0 & 0 & 1 & 0 \\
 \hline
 $D_{p+1}$ & $(p-1)/2$ & 1 & 0 & 0 & 1 \\
 \hline
 $\rm{A}_4$ & $(p-1)/4$ & $(p+1)/3$ & 0 & 0 & 0 \\
 \hline \hline
 $|\SolG{x}|$ & $(p-1)(2p-1)$ & $4(p+1)$ & $p(p-1)/2$ & $p(p-1)$& $p+1$\\
 \hline
\end{tabular}

\caption{Classification of $\SolG{x}$ for $G=\PSL(2, p)$ when $p\equiv 5\pmod{24}$}
    \label{tbl2,p,5}
\end{table}

\begin{table}[H]

\begin{tabular}{| c || c | c | c | c | c | c |} 
 \hline
 \quad Maximal Subgroup & $|x| = 2$ & $|x| = 3$ & $|x| = 4 $ & $|x| = p$ & $|x| \divides p-1$ & $|x|\divides p+1$  \\ [0.5ex] 
 \hline\hline
 $\cp$ & 0 & 2 & 0 & 1 & 2 & 0 \\ 
 \hline
 $D_{p-1}$ & $(p+1)/2$ & 1 & 0 & 0 & 1 & 0 \\
 \hline
 $D_{p+1}$ & $(p+3)/2$ & 0 & 1 & 0 & 0 & 1 \\
 \hline
 $\rm{S}_4$ & $3(p+1)/4$ & $(p-1)/3$ & $(p+1)/4$ & 0 & 0 & 0 \\
 \hline \hline
 $|\SolG{x}|$ & $(p+1)(p+4)$ & $(p-1)(p+6)$ & $5(p+1)$ & $p(p-1)/2$ & $p(p-1)$& $p+1$\\
 \hline
\end{tabular}
\caption{Classification of $\SolG{x}$ for $G=\PSL(2, p)$ when $p\equiv 7\pmod{24}$}
    \label{tbl2,p,7}
\end{table}

\begin{table}[H]

\begin{tabular}{| c || c | c | c | c | c |} 
 \hline
 \quad Maximal Subgroup & $|x| = 2$ & $|x| = 3$ & $|x| = p$ & $|x| \divides p-1$ & $|x|\divides p+1$  \\ [0.5ex] 
 \hline\hline
 $\cp$ & 0 & 0 & 1 & 2 & 0 \\ 
 \hline
 $D_{p-1}$ & $(p+1)/2$ & 0 & 0 & 1 & 0 \\
 \hline
 $D_{p+1}$ & $(p+3)/2$ & 1 & 0 & 0 & 1 \\
 \hline
 $\rm{A}_4$ & $(p+1)/4$ & $(p+1)/3$ & 0 & 0 & 0 \\
 \hline \hline
 $|\SolG{x}|$ & $p(p+1)$ & $4(p+1)$ & $p(p-1)/2$ & $p(p-1)$& $p+1$\\
 \hline
\end{tabular}
\caption{Classification of $\SolG{x}$ for $G=\PSL(2, p)$ when $p\equiv 11\pmod{24}$}
    \label{tbl2,p,11}
\end{table}

\begin{table}[H]

\begin{tabular}{| c || c | c | c | c | c |} 
 \hline
 \quad Maximal Subgroup & $|x| = 2$ & $|x| = 3$ & $|x| = p$ & $|x| \divides p-1$ & $|x|\divides p+1$  \\ [0.5ex] 
 \hline\hline
 $\cp$ & 2 & 2 & 1 & 2 & 0 \\ 
 \hline
 $D_{p-1}$ & $(p+1)/2$ & 1 & 0 & 1 & 0 \\
 \hline
 $D_{p+1}$ & $(p-1)/2$ & 0 & 0 & 0 & 1 \\
 \hline
 $\rm{A}_4$ & $(p-1)/4$ & $(p-1)/3$ & 0 & 0 & 0 \\
 \hline \hline
 $|\SolG{x}|$ & $(p-1)(2p-1)$ & $(p-1)(p+3)$ & $p(p-1)/2$ & $p(p-1)$& $p+1$\\
 \hline
\end{tabular}

\caption{Classification of $\SolG{x}$ for $G=\PSL(2, p)$ when $p\equiv 13\pmod{24}$}
    \label{tbl2,p,13}
\end{table}

\begin{table}[H]

\begin{tabular}{| c || c | c | c | c | c | c |} 
 \hline
 \quad Maximal Subgroup & $|x| = 2$ & $|x| = 3$ & $|x| = 4 $ & $|x| = p$ & $|x| \divides p-1$ & $|x|\divides p+1$  \\ [0.5ex] 
 \hline\hline
 $\cp$ & 2 & 0 & 2 & 1 & 2 & 0 \\ 
 \hline
 $D_{p-1}$ & $(p+1)/2$ & 0 & 1 & 0 & 1 & 0 \\
 \hline
 $D_{p+1}$ & $(p-1)/2$ & 1 & 0 & 0 & 0 & 1 \\
 \hline
 $\rm{S}_4$ & $3(p-1)/4$ & $(p+1)/3$ & $(p-1)/4$ & 0 & 0 & 0 \\
 \hline \hline
 $|\SolG{x}|$ & $(p-1)(2p+3)$ & $7(p+1)$ & $(p-1)(p+4)$ & $p(p-1)/2$ & $p(p-1)$& $p+1$\\
 \hline
\end{tabular}
\caption{Classification of $\SolG{x}$ for $G=\PSL(2, p)$ when $p\equiv 17\pmod{24}$}
    \label{tbl2,p,17}
\end{table}

\begin{table}[H]

\begin{tabular}{| c || c | c | c | c | c |} 
 \hline
 \quad Maximal Subgroup & $|x| = 2$ & $|x| = 3$ & $|x| = p$ & $|x| \divides p-1$ & $|x|\divides p+1$  \\ [0.5ex] 
 \hline\hline
 $\cp$ & 0 & 2 & 1 & 2 & 0 \\ 
 \hline
 $D_{p-1}$ & $(p+1)/2$ & 1 & 0 & 1 & 0 \\
 \hline
 $D_{p+1}$ & $(p+3)/2$ & 0 & 0 & 0 & 1 \\
 \hline
 $\rm{A}_4$ & $(p+1)/4$ & $(p-1)/3$ & 0 & 0 & 0 \\
 \hline \hline
 $|\SolG{x}|$ & $p(p+1)$ & $(p-1)(p+3)$ & $p(p-1)/2$ & $p(p-1)$& $p+1$\\
 \hline
\end{tabular}

\caption{Classification of $\SolG{x}$ for $G=\PSL(2, p)$ when $p\equiv 19\pmod{24}$}
    \label{tbl2,p,19}
\end{table}

\begin{table}[H]

\begin{tabular}{| c || c | c | c | c | c | c |} 
 \hline
 \quad Maximal Subgroup & $|x| = 2$ & $|x| = 3$ & $|x| = 4 $ & $|x| = p$ & $|x| \divides p-1$ & $|x|\divides p+1$  \\ [0.5ex] 
 \hline\hline
 $\cp$ & 0 & 0 & 0 & 1 & 2 & 0 \\ 
 \hline
 $D_{p-1}$ & $(p+1)/2$ & 0 & 0 & 0 & 1 & 0 \\
 \hline
 $D_{p+1}$ & $(p+3)/2$ & 1 & 1 & 0 & 0 & 1 \\
 \hline
 $\rm{S}_4$ & $3(p+1)/4$ & $(p+1)/3$ & $(p+1)/4$ & 0 & 0 & 0 \\
 \hline \hline
 $|\SolG{x}|$ & $(p+1)(p+4)$ & $7(p+1)$ & $5(p+1)$ & $p(p-1)/2$ & $p(p-1)$& $p+1$\\
 \hline
\end{tabular}

\caption{Classification of $\SolG{x}$ for $G=\PSL(2, p)$ when $p\equiv 23\pmod{24}$}
    \label{tbl2,p,23}
\end{table}
\end{center}

\begin{proof}
We obtain the isomorphism types of the maximal subgroups of $\PSL(2, p)$ from \cite{King}*{Corollary 2.2}. We note that when $G$ has a subgroup isomorphic to $\rm{S}_4$, there are two conjugacy classes of $\rm{S}_4$ subgroups and two conjugacy classes corresponding to order $4$ elements. In any other case, there is always exactly one conjugacy class of each type of maximal subgroup.
    
First, we note that the last three columns are always the exact same in any of the eight cases and these are very similar to the last three columns of Table \ref{tbl2,3p}. This is because we are assuming $|x|>4$ in these last three columns so exactly the same proof applies since $x\notin \rm{A}_4$ or $\rm{S}_4.$ In addition, by  \cite{King}*{Theorem 2.1} we can see that when $p\equiv \pm 3\pmod 8$, there are elements of order $4$, but they behave in the same way as the $|x|\divides p\pm 1$ elements, so we do not need a separate column for them. The $|x|=p$ column follows easily as $p$ just divides the order of the maximal subgroup $\cp$ and then we apply Lemma \ref{formula2} and the fact that the number of conjugates of maximal subgroup $M$ in $G$ is $[G:M]$, as before.


The $|x|=3$ and $|x|=4$ cases are similar to each other and depend on the value of $p\pmod 3$ and $p\pmod 4$ respectively. First, by the possibilities for $p$, we see that if $x$ is an element of order $3$ or $4$, then either $|x|$ divides $p-1$ or $|x|$ divides $p+1.$ In both of these cases, we may actually apply the exact same logic we did for the last two columns of the table to construct the first three rows of the table. Unlike the last two columns in which $|x|>4,$ we must also consider $x$ being contained in some copies of $\rm{A}_4$ or $\rm{S}_4.$ Note that if $p\equiv \pm 3\pmod 8$, then an element of order $4$ has the same solubilizer as an element of order greater than $4$ dividing $p\mp 1.$ This is because $G$ has $\rm{A}_4$ maximal subgroups (instead of $\rm{S}_4$) and $\rm{A}_4$ contains no elements of order $4$.

We will look at the specific case $p\equiv 17\pmod{24}$ in detail to get an idea of how this works. In this case, $p\equiv -1\pmod 3$ and $p\equiv 1\pmod 4$. So for $|x|=3$,  $x$ is contained in zero copies of $\cp$ and $D_{p-1}$, and one copy of $D_{p+1}$ just like any other non-involution element of order dividing $p+1.$ For $|x|=4$, $x$ is contained in two copies of $\cp,$ one copy of $D_{p-1}$, and zero copies of $D_{p+1}$ just like any other non-involution element of order dividing $p-1.$  

Then, we note that the number of copies of $\rm{S}_4$ in $G$ is $[G:N_G(M)]=[G:M]$ where $M$ is a maximal subgroup of $G$ isomorphic to $\rm{S}_4$. There are two conjugacy classes of maximal subgroups of $G$ isomorphic to $\rm{S}_4$, each of which having $\frac{p(p+1)(p-1)}{48}$ conjugate subgroups, giving a total number of $\frac{p(p+1)(p-1)}{24}.$ We also know from \cite{King}*{Theorem 2.1} that there are $p(p-1)/2$ cyclic subgroups of order $3$ and $p(p+1)/2$ cyclic subgroups of order $4$ in $G$. A subgroup isomorphic to either $C_3$ or $C_4$ has two elements of order $3$ or $4$ respectively. So there are $p(p-1)$ elements of order $3$ and $p(p+1)$ elements of order $4$ in $G$. Note that these order $4$ elements are split into two equal-sized conjugacy classes (based on which of the conjugacy classes of the $\rm{S}_4$ subgroups they are contained in). We also note that $\rm{S}_4$ has eight elements of order $3$ and six elements of order $4.$

Now, we can use Lemma \ref{formula} as before to calculate the number of copies of $\rm{S}_4$ that $x$ is contained in. If $|x|=3,$ then the number of subgroups isomorphic to $\rm{S}_4$ containing $x$ is:
\[\frac{8\cdot p(p+1)(p-1)/24}{p(p-1)}=\frac{p+1}{3},\]
and if $|x|=4,$ then the number of subgroups isomorphic to $\rm{S}_4$ containing $x$ is:
\[\frac{6\cdot p(p+1)(p-1)/24}{p(p+1)}=\frac{p-1}{4}.\]
This completes every column of Table \ref{tbl2,p,17} except the first. 

This approach works for all of the $8$ cases mod $24$ that can arise, giving slightly different results for each. If $p\equiv \pm 3\pmod 8$ and $|x|=3$ or $p\equiv \pm 1\pmod 8$ and $|x|=4$, then $|x|\divides p \pm 1$ and so  $p\equiv \mp 1 \pmod{|x|}$. Then there are 
\[\frac{(p\pm 1)}{|x|}\]
copies of $\rm{A}_4$ when $|x|=3$ or $\rm{S}_4$ when $|x|=4$. We summarize this all in the Tables \ref{tbl2,p,1} through \ref{tbl2,p,23}.

 Next, we seek to describe the intersections between these copies of $\rm{A}_4$ or $\rm{S}_4$ that contain $x$ and the other maximal subgroups containing $x$. As stated above, the intersections between the maximal subgroups in the first three rows of the table are the same as for the $|x|\divides p\pm 1, |x|>4$ cases so this is all we are left to consider. As discussed in Theorems \ref{2p} and \ref{3p}, the intersections always contain subgroups isomorphic to $C_{|x|},$ namely the subgroup $\gen{x}.$  We note that $C_3$ is maximal in $\rm{A}_4$, so we must always get subgroups isomorphic to $C_3$ when intersecting any copy of $\rm{A}_4$ containing $x$ with another maximal subgroup containing $x$ for $|x|=3.$ However, in $\rm{S}_4,$ both $C_3$ and $C_4$ are contained in dihedral groups of order $6$ and $8$ respectively, so if $p\equiv \pm 1 \pmod 8,$ we have a bit more work to do. 

Again, we shall look at the $p\equiv 17\pmod{24}$ case and the result generalizes nicely. First, consider $|x|=3$. As $3\divides p+1$, the $D_{p+1}$ subgroup contains subgroups isomorphic to $D_6.$ In particular, if we pick $y$ to be any involution in $D_{p+1}$, then we have $\gen{x,y}\cong D_6 \le D_{p+1}.$ We can see that these $D_6$ subgroups of $D_{p+1}$ partition the involutions of $D_{p+1}$ into $(p+1)/6$ sets of size $3.$ 

Now, all of these $D_6$ subgroups of $G$ are conjugate so each one is contained in the same number of maximal subgroups isomorphic to $\rm{S}_4$. Namely, there are $(p+1)/3$ subgroups isomorphic to $\rm{S}_4$ containing $x$ and these must contain $(p+1)/6$ subgroups isomorphic to $D_6$. So each $D_6$ must be contained in exactly two copies of $\rm{S}_4$. Hence, the $\rm{S}_4$s that contain $x$ are split into $(p+1)/6$ pairs depending on which of the $(p+1)/6$ subgroups isomorphic to $D_6$ containing $x$ they contain. The $D_{p+1}$ that contains $x$, contains all of these $D_6$ subgroups and so it has intersection of size $6$ with any of the $\rm{S}_4$s containing $x.$

For $|x|=4,$ we have a similar result. However, when picking an involution $y\in D_{p-1},$ we cannot pick the involution which is in a singleton conjugacy class because in this case $\gen{x,y}\cong C_4$ ($\neq D_8$). Then we have $(p-1)/2$ choices for $y$ which gives $(p-1)/8$ subgroups of $D_{p-1}$ isomorphic to $D_8.$ Thus the $\rm{S}_4$s that contain $x$ are split into $(p-1)/8$ pairs depending on which of the $(p-1)/8$ subgroups isomorphic to $D_8$ containing $x$ they contain and the $D_{p-1}$ consisting of $x$ contains all of these $D_8$ subgroups and so it has intersection of size $8$ with any of the $\rm{S}_4$s containing $x.$

This approach works for all of the $8$ cases mod $24$ that can arise, giving slightly different results for each. If $p\equiv \pm 3\pmod 8$, then the intersections of the maximal subgroups containing $x$ are always just $C_{|x|}$ for $|x|=3$ and are $C_{p\pm 1}$ for $|x|=4$. If $p\equiv \pm 1\pmod 8,$ then let $|x|=3$ or $4$ and suppose $|x|\divides p \pm 1.$ Then, the intersection between the subgroup isomorphic to $D_{p\pm 1}$ and each of the $(p\pm 1)/|x|$ subgroups isomorphic to $\rm{S}_4$ containing $x$ is $D_{2|x|}.$ Finally, the $\rm{S}_4$s are split into $(p\pm 1)/(2|x|)$ pairs where each pair has intersection isomorphic to $D_{2|x|}.$ In any other case, the intersection is just $C_{|x|}.$ We summarize this all in the Tables \ref{tbl2,p,1} through \ref{tbl2,p,23}.

Now, we can use the principle of inclusion exclusion (PIE) to count the size of $\SolG{x}$ which is the union of the copies of maximal subgroups containing $x$, for $|x|=3$ or $4.$ This is a simple calculation that follows the same steps we used when $G=\PSL(2,3^p)$ and $|x| =2$. So we omit the details and put the final results in the Tables \ref{tbl2,p,1} through \ref{tbl2,p,23}.

We continue to the $|x|=2$ case. To calculate the number of each isomorphism type of maximal subgroup containing $x$, we use the same method as for the $|x|=3$ or $4$ case. We summarize the technique in the following table where we assume $p\equiv \pm 1 \pmod 4$:

\begin{table}[H]
\begin{center}
\begin{tabular}{| c || c | c | c | c |} 
 \hline
 \quad & $\cp$ & $D_{p-1}$ & $D_{p+1}$ & $\rm{A}_4$ or $\rm{S}_4$ \\ [0.75ex] 
 \hline\hline
 Copies of maximal subgroup in $G$ & $p+1$ & $\frac{p(p+ 1)}{2}$ & $\frac{p(p- 1)}{2}$ & $\frac{p(p+ 1)(p-1)}{24}$ \\ [0.75ex] 
 \hline
 Involutions in maximal subgroup & $\frac{1\pm 1}{2}$ & $\frac{p-1}{2}+\frac{1\pm 1}{2}$ & $\frac{p+1}{2}+\frac{1\mp 1}{2}$ & $3$ or $9$ \\[0.75ex] 
 \hline
 Maximal subgroups containing $x$ & $1\pm 1$ & $\frac{p+1}{2}$ & $\frac{p+1}{2}\mp 1$  & $\frac{p\mp1}{4}$ or $\frac{3(p\mp1)}{4}$\\[0.75ex] 
 \hline
\end{tabular}
\end{center}
\caption{Classification of involutions and subgroups in $G$ when $G=\PSL(2, p)$}
    \label{tblinvols}
\end{table}

The first row gives the number of subgroups of $G$ isomorphic to each type of maximal subgroup, the second gives the number of involutions in each one of these subgroups, and the last uses Lemma \ref{formula} and the number of involutions in $G$ to give the number of maximal subgroups of each type that contain an involution $x$. We use these values to populate the first column of Tables \ref{tbl2,p,1} to \ref{tbl2,p,23}. Now, based on the definition of $\SolG{x}$, it would seem natural to once again look at the intersections between these. However, for involutions this becomes very difficult. So we will find $|\SolG{x}|$ using Lemma \ref{invol} instead.

To apply this lemma, consider the following. Define $G_i:=\{g\in G \mid |g|=i\}$ for $i=1, 2, 3, 4,$ or $p$ and let $G_{p\pm1}:=\{g\in G \mid |g| \divides p\pm 1\}.$ Then, as we have shown above, for any $i\in \{1, 2 , 3, 4, p, p\pm1\}$, $\SolG{y}$ has the same structure for any $y\in G_i$. In particular, we have partitioned $G$ into seven sets based on the order of the elements for which the number of involutions in $\SolG{y}$ is constant on each set. So we want to find how many elements are in each of these seven sets and how many involutions are in $\SolG{y}$ for a $y$ in each one. We will sum the product of these over the seven sets and then divide by the number of involutions in $G$ to get $|\SolG{x}|$ for an involution.

First, we will to count the number of elements in each of the $G_i$. We shall specifically look at the case $p\equiv 17 \pmod{24}$ but the other cases are very similar (and in fact even easier). In this case, $3\divides p+1$ and $4\divides p-1$ so we know that there are $p(p-1)/2$ and $p(p+1)/2$ cyclic subgroups isomorphic to $C_3$ and $C_4$ respectively which gives $p(p-1)$ and $p(p+1)$ elements of order $3$ and $4$ respectively in $G$. Further, there are $p(p+1)/2$ involutions and clearly the identity element. There are $p+1$ cyclic groups isomorphic to $C_p$ in $G$ which gives $(p+1)(p-1)=p^2-1$ elements of order $p$.

For the elements of order dividing $p\pm 1$ but greater than $4,$ we note that $|x|\divides (p\pm 1)/2$ since these elements are contained in $C_{(p\pm1)/2}$ subgroups. So again, we look at the $p(p\mp 1)/2$ cyclic subgroups of order $(p\pm 1)/2$. To count the number of elements with order greater than $4,$ we subtract the number of elements of order $1$, $2$, $3$, or $4$ from $|C_{(p\pm1)/2}|$. There is the one identity element. In this case, since $3\divides p+1,$ there are two order $3$ elements in $C_{(p+ 1)/2}$ and zero in $C_{(p-1)/2}$. Since $4 \divides p-1,$ there is an order $2$ element in $C_{(p-1)/2}$ and zero in $C_{(p+1)/2}$. Finally, since $8 \divides p-1,$ there are two order $4$ elements in $C_{(p-1)/2}$ and zero in $C_{(p+1)/2}$.

Hence, we get that each $C_{(p-1)/2}$ contains $(p-1)/2-4$ elements of order greater than $4$ while each $C_{(p+1)/2}$ contains $(p+1)/2-3$ elements of order greater than $4.$ So the total number of elements of order dividing $p-1$ (resp. $p+1$) but greater than $4$ in $G$ is $p(p+1)/2\cdot ((p-1)/2-4)$ (resp. $p(p-1)/2\cdot ((p+1)/2-3)$.)

Now, we want to calculate the number of involutions in $\SolG{y}$ for an element $y$ of each order, again looking specifically at the $p\equiv 17\pmod{24}$ case. If $y$ is itself an involution, we use the fact that the subgroup generated by two involutions is always either a cyclic group or a dihedral group, both of which are soluble. Hence, the number of involutions in $\SolG{y}$ is just the total number of involutions in $G$, namely $p(p+1)/2$. Similarly, if $y$ is the identity, every involution is in $\SolG{y}$. Next, we seek to count the number of involutions in $\SolG{y}$ for $|y|=3$ or $4$. 

Recreating the important columns of Table \ref{tbl2,p,17} for ease of reference, we have:
\begin{table}[H]
\begin{center}
\begin{tabular}{| c || c | c |} 
 \hline
 \quad Maximal Subgroup & $|y| = 3$ & $|y| = 4 $ \\ [0.5ex] 
 \hline\hline
 $\cp$ & 0 & 2  \\ 
 \hline
 $D_{p-1}$ &  0 & 1 \\
 \hline
 $D_{p+1}$ &  1 & 0 \\
 \hline
 $\rm{S}_4$& $(p+1)/3$ & $(p-1)/4$ \\
 \hline \hline
 $|\SolG{y}|$ & $7(p+1)$ & $(p-1)(p+4)$ \\
 \hline
\end{tabular}
 \end{center}
 \caption{Classification of $\SolG{y}$ for $|y|=3, 4$ for $G=\PSL(2, p)$ when $p\equiv 17 \pmod{24}$}
    \label{tblsmall}
\end{table}

 If $|y|=3$. There are nine involutions in each of the $(p+1)/3$ copies of $\rm{S}_4$ containing $y$ and $(p+1)/2$ involutions in $D_{p+1}$ since $4\ndivides p+1$ so we do not have a rotational involution. So there are $(p+1)/2+9\cdot (p+1)/3$ total involutions. Now, we look at the intersections. When an intersection is isomorphic to $C_3,$ it contains no involutions so we may ignore it. Then, there are $(p+1)/3$ intersections between $D_{p+1}$ and $\rm{S}_4$ which are isomorphic to $D_6$, there are $(p+1)/6$ intersections between two $\rm{S}_4$s isomorphic to $D_6$, and there are $(p+1)/6$ three-way intersections isomorphic to $D_6$. These latter two cancel out using PIE. In  $D_6,$ there are three involutions so we must subtract $3\cdot (p+1)/3=p+1$ from the initial calculation which gives $5(p+1)/2$ total involutions in $\SolG{y}$ for $|y|=3.$

 Next let $|y|=4$. Again, there are nine involutions in each copy of $\rm{S}_4,$ and $(p-1)/2+1=(p+1)/2$ involutions in $D_{p-1}$ since $4 \divides p-1$ as we have a rotational involution. Further, since $4\divides p-1$ and so $2\divides (p-1)/2$, there are $p$ involutions in $\cp.$ Adding these all up gives $p\cdot 2 + (p+1)/2\cdot 1+9\cdot (p-1)/4$ total involutions. As described above, we now have intersections either isomorphic to $C_{(p-1)/2}$ which has one involution, $C_4$ which has one involution, or $D_8$ which has five involutions. We shall first assume that every intersection is one of the first two and then subtract off the extra four involutions for each of the intersections which are $D_8.$

 In total, there are $(p-1)/4+1+2$ maximal subgroups that contain $y$, so applying PIE as before
 \[\sum_{i=2}^{(p-1)/4+3} (-1)^{i+1}\binom{(p-1)/4+3}{i}=-\binom{(p-1)/4+3}{1}+1\]
 by rearranging the expression for $(1-1)^n$. As before in the case $|y|=3$, there are $(p-1)/4$ intersections between $D_{p-1}$ and $\rm{S}_4$ isomorphic to $D_8,$ there are $(p-1)/8$ intersections between two $\rm{S}_4$s isomorphic to $D_8,$ and $(p-1)/8$ three-way intersections isomorphic to $D_8.$ Again, these last two cancel when doing PIE so we need to subtract off $4\cdot (p-1)/4$ from our sum to account for these dihedral group intersections. Overall, we thus get that the number of involutions in $\SolG{y}$ for $|y|=4$ is: $$|I(\SolG{y})|=2p+(p+1)/2+9(p-1)/4-(p-1)/4-3+1-(p-1)=(7p-5)/2.$$ 
 
Thus, the number of involutions in $\SolG{y}$ for $|y|=3, 4$ is $5(p+1)/2$ and $(7p-5)/2$ respectively. For $|y|=p,$ we have $\SolG{y}=\cp$ which contains one involution since $(p-1)/2$ is even. For $|y|\divides p-1,$ we can easily use PIE to see that there are $(5p-3)/2$ involutions in $\SolG{y}$ and finally, for $|y|\divides p+1,$ there are $(p+1)/2$ involutions in $\SolG{y}.$

We can now use Lemma \ref{invol} to compute $|\SolG{x}|$ for $x$ an involution, nothing that the number of involutions in $G$ is $p(p+1)/2$.
\begin{align*}
    |\SolG{x}| &= \frac{2}{p(p+1)} \Bigg[ 1\cdot\frac{p(p+1)}{2}+\frac{p(p+1)}{2}\cdot\frac{p(p+1)}{2} + p(p-1)\cdot \frac{5(p+1)}{2} + p(p+1) \cdot \frac{(7p-5)}{2} \\
    &+ (p^2-1)\cdot p + \frac{p(p+1)}{2}\left( \frac{p-1}{2}-4\right)\cdot \frac{5p-3}{2} + \frac{p(p-1)}{2}\left( \frac{p+1}{2}-3\right)\cdot \left(\frac{p+1}{2}\right) \Bigg]
    \\
    &= 2p^2+p-3=(p-1)(2p+3). \\
    &
\end{align*}

We can follow the exact same procedure for each of the other cases to compute $|\SolG{x}|$ which we put in the Tables \ref{tbl2,p,1} through \ref{tbl2,p,23}.
    
\end{proof}

\begin{Theorem}\label{sz}
    Let $G=\Sz(2^p)$ where $p$ is an odd prime. Then the following table classifies all subsets $\SolG{x}$ for $x\in G$ based on the order of $x$ in terms of the number of maximal subgroups of each isomorphism type contained in $\SolG{x}$ and $|\SolG{x}|.$ For convenience, we define $q_\pm :=q\pm \sqrt{2q}+1.$
\end{Theorem}

\begin{table}[H]
\begin{center}
\begin{tabular}{| c || c | c | c | c | c |} 
 \hline
 \quad Maximal Subgroup & $|x| = 2$ & $|x| = 4$ & $|x| \divides q-1$ & $|x| \divides q_+$ & $|x|\divides q_-$  \\ [0.5ex] 
 \hline\hline
 $\cccc$ & 1 & 1 & 2 & 0 & 0 \\ 
 \hline
 $D_{2(q-1)}$ & $q^2/2$ & 0 & 1 & 0 & 0 \\
 \hline
 $C_{q_+}\rtimes C_4$ & $q^2/4$ & $q/2$ & 0 & 1 & 0 \\
 \hline
 $C_{q_-}\rtimes C_4$ & $q^2/4$ & $q/2$ & 0 & 0 & 1 \\
 \hline \hline
 Intersections & $\cong C_2$ or $C_4$ & $\cong C_4$ & $\cong C_{q-1}$ & N/A & N/A \\
 \hline
 $|\SolG{x}|$ & $q^2(4q-3)$ & $q^2(q+3)$ & $2q^2(q-1)$ & $4q_+$ & $4q_-$\\
 \hline
\end{tabular}
\end{center}
\caption{Classification of $\SolG{x}$ for $G=\Sz(2^p)$}
    \label{tblsz,2p}
\end{table}

\begin{proof}
    We obtain the classification of the isomorphism types of the maximal subgroups of $G$ and that there is only one conjugacy class of each type of maximal subgroup from \cite{Abe}*{Proposition 2} and \cite{Bray}*{Theorem 7.3.3}. We also get from \cite{Wujie}*{Theorem 2} that the set of orders of the elements of $G$ consists of $2$, $4$, and all factors of $q-1$, and $q_\pm$. The last three columns of Table \ref{tblsz,2p} follow in the exact same way as in the proof of Theorem \ref{3p}. We also calculate the number of copies of the maximal subgroups that contain an element of order $2$ or $4$  using Lemma \ref{formula} in the same way as in the proof of Theorem \ref{2p} after calculating the number of elements of order $2$ and $4$ in $G$.
    
    Any element of order $2$ or $4$ is contained in a Sylow $2$-subgroup of $G$ which is isomorphic to $C_2^p. C_2^p$. We know this group has $q-1$ involutions and $q(q-1)$ elements of order $4$. Now, we see from \cite{Suzuki}*{Theorem 1} that all of these Sylow subgroups have pairwise trivial intersections and using $[G:N_G(C_2^p. C_2^p)]=[G:C_2^p. C_2^p]$, we can find that there are $q^2+1$ such subgroups. Hence, there are $(q^2+1)(q-1)$ and $q(q^2+1)(q-1)$ elements of order $2$ and $4$ in $G$ respectively.

    The intersections in the $|x|=4$ case are all isomorphic to $C_4$ since $C_4$ is maximal in $C_{q_\pm}\rtimes C_4.$ Then, $|\SolG{x}|=q^2(q+3)$ follows quickly by applying PIE as in the proof of Theorem \ref{2p}. For $|x|=2$, the intersections can be either $C_2$ or $C_4$ and we approach this analysis as follows. First, since $q-1$ is odd, $C_2$ is maximal in $D_{2(q-1)}$ and so the intersection of $D_{2(q-1)}$ with any of the other maximal subgroups is just $C_2$. 

    To calculate the intersections of the other maximal subgroups containing $x$, we first need to figure out how many subgroups isomorphic to $C_4$ contain each involution. Since all of the involutions in $G$ are conjugate, they behave in the same way. In particular, any involution, $x\in G$, must have the same number of elements $y$ of order $4$ such that $y^2=x$. From above, there are $(q^2+1)(q-1)$ involutions and $q(q^2+1)(q-1)$ elements of order $4$ in $G$. As for any element $y$ of order $4$, $y^2$ is an involution, we can find that the number of elements $y$ of order $4$ such that $y^2=x$ is $\frac{q(q^2+1)(q-1)}{(q^2+1)(q-1)}=q$. This gives $q/2$ subgroups isomorphic to $C_4$ containing $x$ as each $C_4$ contains two distinct elements of order $4$ that square to $x$. All of these subgroups are clearly contained in the subgroup $\cccc$ consisting of $x$ because it contains the Sylow $2$-subgroup consisting of $x$.

    There is only one conjugacy class containing all of the involutions in $C_{q_+}\rtimes C_4$. Each involution is contained in a unique $C_4$ in $C_{q_+}\rtimes C_4.$ Thus, for a given $x\in G$ and a given $C_{q_+}\rtimes C_4$ containing $x$, there is exactly one $C_4\le C_{q_+}\rtimes C_4$ that contains $x.$ Further, we know that all of the $q/2$ subgroups isomorphic to $C_4 $ which contain $x$ are conjugate in $G$ so any specific $C_4$ cannot be contained in the more copies of $C_{q_+}\rtimes C_4$s. Indeed, each $C_4$ that contains $x$ must be contained in exactly $(q^2/4)/(q/2)=q/2$ of the $C_{q_+}\rtimes C_4$s that contain $x$.
    
    Any two of these groups that contain the same $C_4$ clearly have intersection isomorphic to this $C_4$ since $C_4$ is maximal so it cannot be larger. And two groups that contain different $C_4$s must only intersect at $\gen{x}\cong C_2.$ Hence, we have split the $q^2/4$ copies of $C_{q_+}\rtimes C_4$ into $q/2$ sets of each $q/2$ copies, where within each set, the intersection is $C_4$ and between sets, it is just $C_2.$

We apply similar logic to the $C_{q_-}\rtimes C_4$ subgroups that contain $x$ and get the same result. In addition, we see that these sets actually overlap in the sense that a given copy of $C_4$ containing $x$, is contained in exactly $q/2$ copies of both $C_{q_+}\rtimes C_4$ and $C_{q_-}\rtimes C_4$.

Hence, summarizing the above analysis, we have the following figure:
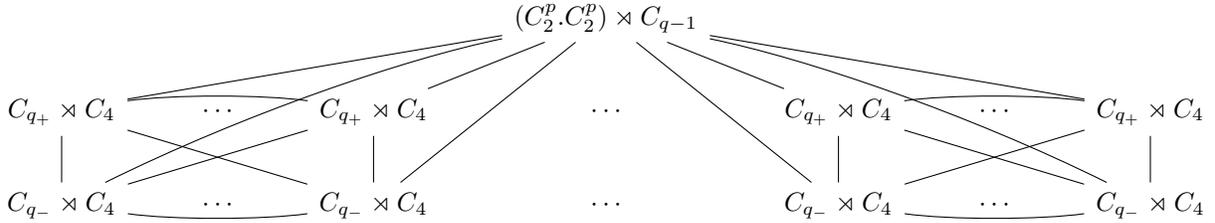
\begin{figure}[H]
\[
\begin{tikzcd}
	&&& \cccc \\
	{C_{q_+}\rtimes C_4} & \cdots & {C_{q_+}\rtimes C_4} & \cdots & {C_{q_+}\rtimes C_4} & \cdots & {C_{q_+}\rtimes C_4} \\
	{C_{q_-}\rtimes C_4} & \cdots & {C_{q_-}\rtimes C_4} & \cdots & {C_{q_-}\rtimes C_4} & \cdots & {C_{q_-}\rtimes C_4}
	\arrow[no head, from=1-4, to=2-1]
	\arrow[no head, from=1-4, to=2-3]
	\arrow[no head, from=1-4, to=2-5]
	\arrow[no head, from=1-4, to=2-7]
	\arrow[curve={height=10pt}, no head, from=1-4, to=3-1]
	\arrow[no head, from=1-4, to=3-3]
	\arrow[no head, from=1-4, to=3-5]
	\arrow[curve={height=-10pt}, no head, from=1-4, to=3-7]
	\arrow[no head, from=2-1, to=3-1]
	\arrow[no head, from=2-3, to=3-3]
	\arrow[no head, from=3-1, to=2-3]
	\arrow[no head, from=2-1, to=3-3]
	\arrow[curve={height=6pt}, no head, from=3-1, to=3-3]
	\arrow[curve={height=-6pt}, no head, from=2-1, to=2-3]
	\arrow[curve={height=6pt}, no head, from=3-5, to=3-7]
	\arrow[no head, from=2-5, to=3-5]
	\arrow[no head, from=3-5, to=2-7]
	\arrow[no head, from=2-5, to=3-7]
	\arrow[no head, from=2-7, to=3-7]
	\arrow[curve={height=-6pt}, no head, from=2-5, to=2-7]
\end{tikzcd}
\]
\caption{$C_4$ intersections of maximal subgroups containing an involution in $\Sz(2^p)$}
    \label{fig2}
\end{figure}
where lines between subgroups represent when the intersection is not just $C_2$, and there are $q/2$ complete blocks containing $q/2$ copies each of $C_{q_+}\rtimes C_4$ and $C_{q_-}\rtimes C_4.$ Finally, as in the calculation of $|\SolG{x}|$ for $|x|=2$ in Theorem \ref{3p}, we can use PIE and subtract off all of the extra intersections we get where the intersections have size $4$ instead of $2$. We omit the calculation here and merely report that the union of these maximal subgroups has size \[|\SolG{x}|=q^2(4q-3).\] 
    
\end{proof}

\begin{Theorem}\label{33}
    Let $G=\PSL(3,3)$. Then the following table classifies all subsets $\SolG{x}$ for $x\in G$ based on the order of $x$ in terms of the number of maximal subgroups of each isomorphism type contained in $\SolG{x}$ and $|\SolG{x}|.$
\end{Theorem}
\begin{table}[H]
\begin{center}
\begin{tabular}{| c || c | c | c | c | c | c | c |} 
 \hline
 \quad Maximal Subgroup & $|x| = 2$ & $|x|=3$  & $|x|=3$   & $|x|=4$ & $|x|=6$ & $|x|=8$ & $|x|=13$ \\ [0.5ex] 
 \hline\hline
 $(C_3^2\rtimes Q_8)\rtimes C_3$ & 10 & 2 & 8 & 2 & 4 & 2 & 0\\ 
 \hline
 $C_{13}\rtimes C_3$ & 0 & 6 & 0 & 0 & 0 & 0 & 1 \\
 \hline
 $\rm{S}_4$ & 18 & 3 & 2 & 2 & 0 & 0 & 0 \\
 \hline \hline
 $|\SolG{x}|$ & 2832 & 1026 & 2376 & 848 & 1368 & 816 & 39 \\
 \hline
$|N_G(\gen{x})|$ & 48 & 18 & 108 & 16 & 12 & 16 & 39\\

 \hline
\end{tabular}
\end{center}
\caption{Classification of $\SolG{x}$ for $G=\PSL(3,3)$}
    \label{tbl3,3}
\end{table}

    \begin{proof}
        We used GAP \cite{GAP} to directly compute how many copies of each maximal subgroup contain an element in each conjugacy class. The intersections are quite complicated so we leave their classification to the appendix but this was also just computed in GAP. Note that this is the only example of a minimal simple group for which $\SolG{x}$ depends on more than just the order of $x$ so we clarify that by distinguishing between these columns with the order of $N_G(\gen{x})$.
    \end{proof}
\end{section}

\begin{section}{Graphical Properties}\label{graphprop}
In this section, we will focus on the solubility graphs associated to the finite groups. In \cites{Akbari 2, Bhowal, Burness}, many properties of this graph have been studied. We prove some properties regarding Eulerian cycles, Hamiltonian cycles and the colorability of this graph.

We first recall that the {\em solubility graph} of $G$, denoted $\sgraph$, is defined as a graph whose vertex set is $G$, and there is an edge between $x$ and $y$ when $\gen{x,y}$ is soluble. As discussed in the introduction [Section \ref{Introduction}], the soluble radical of $G$, $R(G)$, coincides with the set of universal vertices of the graph $\Gamma_S(G)$, namely the vertices that are adjacent to every other vertex of the graph. As a result of Thompson's Theorem, given a soluble group $G$, the solubility graph $\Gamma_S(G)$ is complete. So, to explore the graphical properties we consider only insoluble groups $G$ and focus on the induced solubility graph of $G$ which is the induced subgraph of $\sgraph$ on the set $G \ \backslash \ R(G)$. We denote this graph by $\Delta_S(G)$.

Now, we recall some relevant definitions in graph theory.
A graph $\Gamma$ is called {\em Eulerian} if it contains an Eulerian cycle which is a trail that starts and ends at the same vertex and passes through each edge of $\Gamma$ exactly once. A graph $\Gamma$ is {\em Hamiltonian} if it contains a Hamiltonian cycle which is a cycle visiting each vertex of $\Gamma$ exactly once. The {\em chromatic number} of a graph $\Gamma$, denoted $\chi(\Gamma)$, is the minimal number of colors required to color the vertices such that no two adjacent vertices are the same color.
\begin{Lemma}\label{NonEulerian}
    Given an insoluble finite group $G$, $\Gamma_S(G)$ is not Eulerian. 
\end{Lemma}
\begin{proof}
As a well-known result in graph theory, a connected graph is Eulerian if and only it has no graph vertices of odd degree. As the solubility graph $\Gamma_S(G)$ is connected, we can use this equivalent condition. Assume that $\Gamma_S(G)$ is Eulerian. By definition of the solubility graph, $d_G(x)=|\SolG{x}|-1$. Let $x$ be an element of $R(G)$. Then we have $|\SolG{x}|=|G|=2n+1$ where $n$ is a positive integer. Thus, $G$ has odd order which implies that $G$ is soluble by the Feit-Thompson theorem, a contradiction.
\end{proof}
\begin{Lemma}\label{gamham}
    For a finite group $G$, if $\Delta_S(G)$ has a Hamiltonian cycle, then $\Gamma_S(G)$ does too.
\end{Lemma}
\begin{proof}
    Any element in $R(G)$ is connected to all other elements in the graph $\Gamma_S(G)$. Hence, we can perform the Hamiltonian cycle in $\Delta_S(G)$ except for the final edge, then progress through all of the elements in $R(G)$, before ending where we started to give a Hamiltonian cycle in $\Gamma_S(G).$
\end{proof}
The converse of Lemma \ref{gamham} need not be true since $\Gamma_S(G)$ can have many more connections than $\Delta_S(G)$.
\begin{Lemma}\label{gamcol}
    Let $G$ be a finite group. Then $\chi(\Gamma_S(G))=\chi(\Delta_S(G))+|R(G)|.$
\end{Lemma}
\begin{proof}
This is easy to show using the fact that every element of $R(G)$ is joined to every other element in the group. In fact, given a proper $n$-coloring of $\Delta_S(G)$, we can extend this to a proper $(n+|R(G)|)$-coloring of $\Gamma_S(G)$ by simply coloring each of the elements of $R(G)$ a distinct color that we have not used so far for $\Delta_S(G)$. It is also clear that we cannot obtain a proper coloring with fewer colors.
\end{proof}

As we can see in Lemmas \ref{NonEulerian}, \ref{gamham}, and \ref{gamcol}, we can learn a lot about $\Gamma_S(G)$ from the induced solubility graph $\Delta_S(G)$. Thus, we presently turn our attention to this graph for the rest of our analysis in this paper. 
\begin{Theorem}\label{eul}
     Given an insoluble finite group $G$, $\dgraph$ is Eulerian if and only if $|R(G)|$ is odd and for any element $x$ in $G$, $|\text{Sol}_G(x)|$ is even.
\end{Theorem}
\begin{proof}
We first note that according to \cite{Akbari 2}*{Corollary 2.2}, the induced solubility graph $\dgraph$ is always connected. 
Then, the reverse direction of the statement is easy to show using the definition of $\dgraph$ since $d_\Delta(x)=|\SolG{x}|-|R(G)|-1$ is even for any vertex in $G\setminus R(G)$ where $d_\Delta(x)$ is the degree of $x$. To prove the forward direction, assume $d_\Delta(x)$ is even for all $x\in G\setminus R(G)$. Then, if $|R(G)|$ is even, so too is $|\SolG{x}|$ by Lemma \ref{normalizer}, a contradiction to $d_\Delta(x)$ being even. So $|R(G)|$ is odd and then using $|\SolG{x}|=d_\Delta(x)+|R(G)|+1$, we clearly see that $|\SolG{x}|$ must be even for any $x\in G\setminus R(G).$ Finally, for an element $x\in R(G),$ we have $|\SolG{x}|=|G|$ is even as $G$ is insoluble.
\end{proof}

\begin{Theorem}
    If $G=\PSL(2, 2^p)$ for $p$ a prime, then $\Delta_S(G)$ has a Hamiltonian cycle.
\end{Theorem}

\begin{proof}

We will start and end our cycle at an involution. The subgroup generated by any two involutions is a dihedral group which is soluble. So the involutions of $G$ form a complete subgraph of $\Delta_S(G).$ Also by the proof of the characterization of the solubilizers in Theorem \ref{2p}, we get that it suffices that we cover all involutions, copies of $C_{q-1}$, and copies of $C_{q+1}$ since every element is in one of these subgroups. Further, each of these subgroups is a complete subgraph $\Delta_S(G)$ since cyclic groups are always soluble. We reproduce Table \ref{tbl2,2p} here for ease of reference.

\begin{center}
\begin{tabular}{| c || c | c | c|} 
 \hline
 \quad Maximal Subgroup & $|x| = 2$ & $|x| \divides q+1$ & $|x| \divides q-1$  \\ [0.5ex] 
 \hline\hline
 $\cc$ & 1 & 0 & 2  \\ 
 \hline
 $D_{2(q+1)}$ & $q/2$& 1 & 0  \\
 \hline
 $D_{2(q-1)}$ & $q/2$ & 0 & 1   \\
 \hline \hline
 Intersections & $\cong C_2$ & N/A & $\cong C_{q-1}$ \\
 \hline
 $|\SolG{x}|$ & $3q(q-1)$ & $2(q+1)$ & $2q(q-1)$ \\
 \hline
\end{tabular}
\end{center}

We first claim that the subgraph induced by all of the $C_{q-1}$s contains a Hamiltonian path so for this part we simply need to choose a starting and ending involution. We know that there are a total of $q+1$ subgroups of $G$ isomorphic to $\cc$ by \cite{King}*{Theorem 2.1}. From Theorem \ref{2p}, each $C_{q-1}$ is the intersection of two copies of $\cc$ and there are a total of $\frac{q(q+1)}{2}=\binom{q+1}{2}$ pairs of $\cc$. There are also exactly $\frac{q(q+1)}{2}$ copies of $C_{q-1}$ (\cite{King}*{Theorem 2.1}). So each $C_{q-1}$ is uniquely the intersection of two copies of $\cc$. Now, we label all of the copies of $\cc$ as $\rm{A}_1,\rm{A}_2,...,\rm{A}_{q+1}$ in any random order since there are $q+1$ total. Let $B_i=\rm{A}_i\cap \rm{A}_{i+1}$ and we know that every $B_i$ is a copy of $C_{q-1}$ since every intersection is a $C_{q-1}$.

So then we can find an involution in $\rm{A}_1$ and have it connect to all of the copies of $C_{q-1}$ in $\rm{A}_1$ in any order, finishing with an element in $B_1$. Then when we are at $B_1$ we can connect to all of the copies of $C_{q-1}$ in $\rm{A}_2$, finishing with an element in $B_2$. We can continue this process until we finish with all of the copies of $C_{q-1}$ in $\rm{A}_{q+1}$. Note that we skip the copies of $C_{q-1}$ that we have already covered. Further, at each step, we know that we have not visited the $B_i$ before we get to it. This is because $B_i\cong C_{q-1}$ is defined as $\rm{A}_i\cap \rm{A}_{i+1}$ and so it cannot be in any $\rm{A}_j$ with $j<i$ since it is in exactly two $\rm{A}_i$s from Table \ref{tbl2,2p}. Also note that this path would cover every copy of $C_{q-1}$ since each of these is in exactly two $\rm{A}_i$s and we have reached every $\rm{A}_i$. Once we reach $\rm{A}_{q+1}$ we can return to any involution in $\rm{A}_{q+1}$. So then we can use this involution for the remaining path.

Next, we deal with the $C_{q+1}$ subgroups. First, we show that the intersection of any two $D_{2(q+1)}$s is a unique subgroup isomorphic to $C_2$. Let $D$ and $D'$ be two specified subgroups of $G$ isomorphic to $D_{2(q+1)}$. From the classification of the intersections of the maximal subgroups containing involutions in Table \ref{tbl2,2p}, we know that $D\cap D'\le C_2$. We show that the intersection of $D$ with any $D'$ is nontrivial. We first note that as $D_{2(q+1)}$ has $q+1$ involutions and each of these involutions is contained in exactly $q/2$ copies of $D_{2(q+1)},$ we see that there are $(q+1)(q/2-1)$ subgroups isomorphic to $D_{2(q+1)}$ other than $D$ which contain its involutions. This is because two involutions in $D$ cannot both be contained in the same other $D'$ as otherwise $|D\cap D'|> 2$. However, $(q+1)(q/2-1)=q(q-1)/2-1$ is the total number of $D_{2(q+1)}$ subgroups of $G$ other than $D$ by using \cite{King}*{Theorem 2.1}. This means that every $D_{2(q+1)}$ in $G$ contains an involution that is contained in $D$ so the intersection of $D$ with any other $D_{2(q+1)}$ is nontrivial and must be exactly $C_2.$

 Now, consider ordering the $D_{2(q+1)}$s in some arbitrary order, say labelling them as $B_1, B_2, \ldots, B_{q(q-1)/2}$. Since the intersection of any two $D_{2(q+1)}$ is isomorphic to a unique $C_2$, we have $q(q-1)/2-1$ involutions that occurs between $B_i$ and $B_{i+1}$ for $1\le i \le q(q-1)/2$. There are $q^2-1$ total involutions in $G$ and we started with an involution so we may pick any of the remaining $q^2-2-(q(q-1)/2-1)=(q^2+q-2)/2$ to transition from the $C_{q-1}$ to the $C_{q+1}$ subgroups. Then, we go through the $B_i$s as detailed above, completing each $C_{q+1}$ subgroup and then using the unique involution in $B_i\cap B_{i+1}$ to get to the next $D_{2(q+1)}.$ If it turns out that an involution we need to progress from $B_i$ to $B_{i+1}$ was used as the starting point of the whole cycle, we can simply choose a different involution to start that is in $\rm{A}_1$. At the end, we are guaranteed that there are many involutions left to choose from $B_{q(q-1)/2}$. Finally, we can clean up the remaining involutions since they form a complete subgraph of $\Delta_S(G)$, ending on the involution we originally used to start the cycle. So the proof is complete.
\end{proof}
\begin{Theorem}
    If $G = \PSL(2,3^p)$ for $p$ an odd prime or $G = \PSL(2, p)$ for $p$ a prime with $p\equiv 19\pmod{24}$ and $p\equiv 2$ or $3\pmod 5$,  then $\Delta_S(G)$ never has a Hamiltonian cycle.
\end{Theorem}

\begin{proof} 
We show that for $G=\PSL(2,3^p),\Delta_S(G)$ cannot have a Hamiltonian cycle. The proof for $\PSL(2, p), p\equiv 19 \pmod{24}$ is similar. We reproduce Table \ref{tbl2,3p} above classifying the solubilizers of elements in $G$:

\begin{center}
\begin{tabular}{| c || c | c | c | c|} 
 \hline
 \quad Maximal Subgroup & $|x| = 2$ & $|x| = 3$ & $|x| \divides q-1$ & $|x| \divides q+1$  \\ [0.5ex] 
 \hline\hline
 $\ccc$ & 0 & 1 & 2 & 0  \\ 
 \hline
 $D_{q-1}$ & $(q+1)/2$& 0 & 1 & 0  \\
 \hline
 $D_{q+1}$ & $(q+3)/2$ & 0 & 0 & 1  \\
 \hline
 $\rm{A}_4$ & $(q+1)/4$ & $q/3$ & 0 & 0 \\
 \hline \hline
 Intersections & $\cong C_2$ or $C_2^2$ & $\cong C_3$ & $\cong C_{(q-1)/2}$ & N/A \\
 \hline
 $|\SolG{x}|$ & $q(q+1)$ & $q(q+5)/2$ & $q(q-1)$ & $q+1$ \\
 \hline
\end{tabular}
\end{center}

Consider an element $x$ of order dividing $q+1$. From the proof of Theorem \ref{3p}, we see that $x$ is contained in exactly one subgroup isomorphic to $C_{(q+1)/2}$ which is contained in exactly one $D_{q+1}$, and so $\SolG{x}\cong D_{q+1}.$ Then, we notice that the order of an element in $D_{q+1}$ can only be $2$ or some other divisor of $q+1$. Thus, in $\Delta_S(G),$ we have that $x$ is only adjacent to involutions or other elements in the unique $C_{(q+1)/2}$ subgroup that contains $x.$ If we want to make a Hamiltonian cycle in $\Delta_S(G)$, we need to reach every $D_{q+1}$ subgroup at least once such that we reach all of these non-involutions of order dividing $q+1.$ For each $D_{q+1}$ that we visit, we must also visit an involution before and after all of the elements of order dividing $q+1$ because this is the only way to get in to/out of the $D_{q+1}$.

However, we note from before that there are $q(q-1)/2$ involutions in $G$ and there are also exactly $q(q-1)/2$ subgroups isomorphic to $D_{q+1}.$ Thus, there are too few involutions to visit one before and after each $D_{q+1}$ that we visit. Indeed, we would need at least one more involution that $D_{q+1}$ subgroups for this to work. Hence, $\Delta_S(G)$ can never have a Hamiltonian cycle. For $G=\PSL(2, p),$ the exact same logic applies to the $D_{p+1}$ subgroups of which there are $p(p-1)/2$ which is equal to the number of involutions in $G$.
 \end{proof}

\end{section}

\begin{section}{Conjectures and Corollaries}\label{cors}
\begin{Corollary}
    Let $G$ be a minimal simple group. For all $x\in G$, we have $|N_G(\gen{x})|\divides |\SolG{x}|$.
\end{Corollary}
\begin{proof}
    This is an easy check using the structure of the minimal simple groups to find the normalizers of cyclic subgroups and using the values of $|\SolG{x}|$ computed in the classification tables in Section \ref{class}.
\end{proof}
\begin{Conjecture}{\cite{Mousavi}*{Conjecture 3}}
    For a group $G$ and any element $x\in G$, we have $|N_G(\gen{x})|\divides |\SolG{x}|$.
\end{Conjecture}
\textit{Remark:} We have computationally verified this using GAP for all simple groups of order less than $2$ million.

\begin{Corollary}
    If $G$ is a minimal simple group and $x\in G,$ then $|\SolG{x}|\neq p^n$ for an odd prime $p$ and positive integer $n.$
\end{Corollary}
\begin{proof}
    This follows immediately from the tables in Section \ref{class} giving $|\SolG{x}|$ for the minimal simple groups.
\end{proof}
\begin{Corollary}
    Let $G$ be a minimal simple group and $x\in G.$ Then, $|\SolG{x}|=2^n$ if and only if we are in the following case:
          $G=\PSL(2,2^n-1)$ for $n\equiv 3\pmod 4, n>3$ where $2^n-1$ is a Mersenne prime, and $|x|= 2^m$ for $2<m<n$. 
\end{Corollary}
\begin{proof}
    Again, this follows immediately from the tables in Section \ref{class} giving $|\SolG{x}|$ for the minimal simple groups. For instance, if $G=\PSL(2,127)$ and $|x|=8,16,32,$ or $64,$ then $|\SolG{x}|=128.$
\end{proof}
\begin{Conjecture}{\cite{Mousavi}*{Conjecture 2}}
    If $G$ is any finite insoluble group, then $|\SolG{x}|\neq p^n$ for any odd prime $p$ and positive integer $n.$
\end{Conjecture}
\textit{Remark:} We have computationally verified this using GAP for all simple groups of order less than $2$ million. Note that in {\cite{Mousavi}*{Conjecture 1}}, it was also speculated that if $|\SolG{x}|=2^n$, then $\SolG{x}$ is a subgroup of $G$. We computationally find a counterexample. Namely, if $G=\rm{A}_8$ and $x$ is an order $4$ element in the conjugacy class containing $(1,2,3,4)(5,6)$, then $|\SolG{x}|=1024$ but clearly $\SolG{x}$ is not a subgroup of $G$ since $1024\ndivides |G|$.

As mentioned in the introduction, a finite group is soluble if any $2$-generated subgroup is soluble. In \cite{GW}, the authors provided a probabilistic version of this result as follows. Given a finite group $G$, if the probability that two randomly chosen elements of $G$ generate a soluble group is greater than $11/30$, then $G$ is soluble.

Given an element $x$ in a finite group $G$, we introduce the {\em solubilizer probability} of $x$, denoted by  $P_S(x)$, as follows:
\[P_S(x)=\frac{|\SolG{x}|}{|G|}.\]
Given the minimal simple groups, we noticed that in most of the cases for any nontrivial element $x$ other than involutions, $P_S(x)\le \frac{1}{2}$. Indeed, we can state the following corollary:

\begin{Corollary}\label{pt5}
    Let $G$ be a minimal simple group and $x$ be a nontrivial element of $G$. Then $P_S(x)> \frac{1}{2}$ if and only if one of the following cases is satisfied:
    \begin{enumerate}
        \item [{\rm (1)}] $G= \rm{A}_5$, $x$ is any involution, and $P_S(x)=\frac{3}{5},$
        \item [{\rm (2)}] $G= \PSL(2,7)$, $x$ is any involution, and $P_S(x)=\frac{11}{21}$,
        \item [{\rm (3)}] $G= \PSL(3,3)$, $x$ is any involution, and $P_S(x)=\frac{59}{117}.$
    \end{enumerate}
\end{Corollary}
\begin{proof}
  Using all information about $|\SolG{x}|$ obtained in the tables in Section \ref{class}, it is not hard to find the result.
\end{proof}

Using GAP, we computed the solubilizer probability numbers
for the non-abelian simple groups of various orders not exceeding $2$ million. In fact, except for the three minimal simple groups listed in Corollary \ref{pt5}, we have only found one other group with $P_S(x)> \frac{1}{2}$, namely an involution $x$ of the orthogonal group $O(5,3)$ with $P_S(x)=\frac{5}{9}$. So the following conjecture arises: 
\begin{Conjecture}
    Let $G$ be an insoluble group. Then for all elements $x\in G\setminus R(G)$, we have $P_S(x) \leq \frac{3}{5}$. Furthermore, given a nontrivial element $x\in G$ which is not an involution, we have $P_S(x)<\frac{1}{2}$.
\end{Conjecture}
    \begin{Corollary}
    If $G$ is a minimal simple group and $x\in G$, then $\SolG{x}=N_G(\gen{x})$ if and only if one of the following cases is satisfied:
    \begin{enumerate}
        \item [{\rm (1)}] $G=\PSL(2, 2^p)$ where $p$ is a prime and $|x|\divides 2^p+1,$ 
        \item [{\rm (2)}] $G=\PSL(2,3^p)$ where $p$ is an odd prime, and $|x|\divides 3^p+1$ when $|x|\neq 2,$
        \item [{\rm (3)}] $G=\PSL(2, p)$ where $p$ is a prime with $p\equiv 2$ or $3$ $\pmod 5$, and $|x|\divides p+1$ when $|x|> 4,$
        \item [{\rm (4)}] $G=\PSL(2, p)$ where $p$ is a prime with $p\equiv 2$ or $3$ $\pmod 5$, and $|x|= p,$
        \item [{\rm (5)}] $G=\Sz(2^p)$ where $p$ is an odd prime and $|x|\divides 2^p+2^{(p+1)/2}+1,$
        \item [{\rm (6)}] $G=\Sz(2^p)$ where $p$ is an odd prime and $|x|\divides 2^p-2^{(p+1)/2}+1,$
        \item [{\rm (7)}] $G=\PSL(3,3)$ and $|x|=13.$
    \end{enumerate}
\end{Corollary}
\begin{proof}
    This follows immediately from the tables in Section \ref{class}.
\end{proof}

\begin{Corollary}\label{bounds}
    Given a minimal simple group $G$, we obtain the following bounds on the chromatic number of $\Delta_S(G)$:
    \begin{enumerate}
        \item [{\rm (1)}] If $G=\PSL(2, 2^p)$ where $p$ a prime, then $2^{2p}-1\le \chi(\Delta_S(G))\le 3\cdot2^p(2^p-1)-2,$
        \item [{\rm (2)}] If $G=\PSL(2,3^p)$ where $p$ an odd prime, then $\frac{3^p(3^p-1)}{2}\le \chi(\Delta_S(G))\le 3^p(3^p+1)-2,$
        \item [{\rm (3)}] If $G=\PSL(2, p)$ where $p$ is a prime with $p\equiv 2$ or $3$ mod $5$, then $\frac{p(p-1)}{2}\le \chi(\Delta_S(G))\le (p-1)(2p+3)-2,$
        \item [{\rm (4)}] $G=\Sz(2^p)$ where $p$ is an odd prime, then $(2^{2p}+1)(2^p-1)\le \chi(\Delta_S(G))\le 2^{2p}(4\cdot 2^p-3)-2,$
        \item [{\rm (5)}] $G=\PSL(3,3)$, then $431 \le \chi(\Delta_S(G))\le 2830.$
    \end{enumerate}
\end{Corollary}
\begin{proof}
    The lower bounds follow from the number of colors that we need for involutions as they form a complete subgraph. So we require at least as many colors as the the number of involutions. The upper bounds follow from Brook's theorem and the maximum degree of vertices that we computed due to our classifications of the solubilizers in minimal simple groups. The lower bound for $\PSL(3,3)$ comes from the maximal subgroup $(C_3^2\rtimes Q_8)\rtimes C_3$ of size $432$ which is soluble and so forms a complete subgraph of $\Delta_S(G)$ with $431$ vertices.
\end{proof}
We can also use Lemma \ref{eul} to find the minimal simple groups $G$ such that $\Delta_S(G)$ is Eulerian.
\begin{Corollary}
    Let $G$ be a minimal simple group. Then $\Delta_S(G)$ is Eulerian if and only if $G$ is not one of the following cases:
    \begin{enumerate}
        \item [{\rm (1)}] $G=\PSL(2, p)$ for $p>3$ a prime such that $p\equiv 3$ or $7\pmod{20},$ 
        \item [{\rm (2)}] $G=\PSL(3,3).$
    \end{enumerate}
\end{Corollary}
\begin{proof}
    This immediately follows from the tables in Section \ref{class} classifying $|\SolG{x}|$ for minimal simple groups and Theorem \ref{eul}.
\end{proof}
\end{section}

\begin{section}{Computational Results}\label{comp}
All code, matrices and explicit results can be found in the following Github repository. 
\begin{center}
    https://github.com/jchuharski/GroupsOnGraphs
\end{center}
In the following lemma, we computationally improve the bounds on the chromatic number of $\Delta_S(G)$ for some small simple groups which is based on some computational results.
\begin{Lemma}
Given the following groups $G$, the bounds on the chromatic number of $\Delta_S(G)$ are as follows:
    \begin{enumerate}
        \item [{\rm (1)}] If $G=\rm{A}_5=\PSL(2,4),$ then $\chi(\Delta_S(G))=15,$
        \item [{\rm (2)}] If $G=\PSL(2,8),$ then $\chi(\Delta_S(G))=63,$
        \item [{\rm (3)}] If $G=\PSL(2,16),$ then $\chi(\Delta_S(G)) = 255,$
        \item [{\rm (4)}] If $G=\PSL(2,7),$ then $23 \le \chi(\Delta_S(G))\le 28,$
        \item [{\rm (5)}] If $G=\PSL(2,11),$ then $55\le \chi (\Delta_S(G))\le 56,$
        \item [{\rm (6)}] If $G=\PSL(2,13),$ then $91\le \chi (\Delta_S(G))\le 92,$
        \item [{\rm (7)}] If $G=\PSL(2,17),$ then $153\le \chi (\Delta_S(G))\le 155,$
        \item [{\rm (8)}] If $G=\PSL(2,19),$ then $171\le \chi (\Delta_S(G))\le 174,$
        \item [{\rm (9)}] If $G=\rm{A}_6,$ then $45 \le \chi(\Delta_S(G))\le 47,$
        \item [{\rm (10)}] If $G=\rm{A}_7,$ then $105 \le \chi(\Delta_S(G))\le 107,$
        \item [{\rm (11)}] If $G=\PSL(3,3),$ then $431 \le \chi (\Delta_S(G))\le 441.$
    \end{enumerate}
\end{Lemma}
\begin{proof}
    As in Corollary \ref{bounds}, we obtain the lower bounds from the number of involutions in these groups, which we can compute directly using GAP. For the upper bounds, we employ a greedy coloring algorithm which runs through the vertices in the graph and assigns to each the lowest number that it is not adjacent to. While this does not guarantee optimal coloring, it necessarily produces a proper coloring, hence resulting in an upper bound on the chromatic number. In the first three cases, we obtain a proper coloring with the lower bound and thus must have equality. Similar to how we tightened the lower bound for $\PSL(3,3)$ in Corollary \ref{bounds}, for $\PSL(2,7),$ we use the subgroup $\rm{S}_4$ to get a lower bound of $23$ instead of $21$ which is the number of involutions.
\end{proof}
\begin{Lemma}
    We computationally obtain Hamiltonian cycles $\Delta_S(G)$ for the following groups: $\rm{A}_5, \rm{A}_6, \PSL(2,7), $ $\PSL(2,8), \PSL(2,11), \PSL(2,13), \PSL(2,16), \PSL(2,17), \PSL(3,3).$ As proved previously, for $\PSL(2,19)$ this is impossible.
\end{Lemma}
\begin{proof}
    First, we use a GAP program to produce an adjacency matrix for $\Delta_S(G)$ by checking which pairs of elements generate a soluble subgroup of $G$. Then, we use Julia \cite{Julia} to implement a greedy random algorithm to find a Hamiltonian cycle in this graph. At each step, we randomly move to the neighbor with the least number of neighbors that we still need to get to. The output is in the Github linked above and takes the form of a list of indices, corresponding to the non-identity elements of $G$, listed in the order given by the GAP command List(G).
\end{proof}
\end{section}

\begin{center}
 {\sc Acknowledgments}
\end{center}
This research was initiated in a SPUR program at Cornell University. Jake Chuharski, Vismay Sharan and Zachary Slonim would like to thank Olu Olorode for mentoring them and giving valuable insight  through this program.

\begin{center}
 {\sc Appendix}\label{Appendix}
\end{center}
    The intersections between maximal subgroups containing $x$ for $G=\PSL(3,3).$
    
    \noindent\textbf{Case 1:} $|x|=2.$
    \begin{itemize}
        \item [{\rm -}] $(C_3^2\rtimes Q_8)\rtimes C_3\cap (C_3^2\rtimes Q_8)\rtimes C_3\cong (C_3^2\rtimes C_3)\rtimes (C_2^2).$ Occurs 12 times.
\item [{\rm -}] $(C_3^2\rtimes Q_8)\rtimes C_3\cap (C_3^2\rtimes Q_8)\rtimes C_3\cong \rm{S}_3 \times \rm{S}_3.$ Occurs 20 times.
\item [{\rm -}] $(C_3^2\rtimes Q_8)\rtimes C_3\cap (C_3^2\rtimes Q_8)\rtimes C_3\cong GL(2,3).$ Occurs 13 times.
\item [{\rm -}] $(C_3^2\rtimes Q_8)\rtimes C_3\cap \rm{S}_4\cong D_8.$ Occurs 60 times.
\item [{\rm -}] $(C_3^2\rtimes Q_8)\rtimes C_3\cap \rm{S}_4\cong C_2 \times C_2.$ Occurs 72 times.
\item [{\rm -}] $(C_3^2\rtimes Q_8)\rtimes C_3\cap \rm{S}_4\cong \rm{S}_3.$ Occurs 48 times.
\item [{\rm -}] $\rm{S}_4\cap \rm{S}_4\cong D_8.$ Occurs 15 times.
\item [{\rm -}] $\rm{S}_4\cap \rm{S}_4\cong C_2.$ Occurs 96 times.
\item [{\rm -}] $\rm{S}_4\cap \rm{S}_4\cong C_2 \times C_2.$ Occurs 18 times.
\item [{\rm -}] $\rm{S}_4\cap \rm{S}_4\cong \rm{S}_3.$ Occurs 24 times.
    \end{itemize}

    \noindent\textbf{Case 2:} $|x|=3, |N_G(\gen{x})|=18.$
    \begin{itemize}
        \item [{\rm -}] $\rm{S}_4\cap (C_3^2\rtimes Q_8)\rtimes C_3\cong \rm{S}_3.$ Occurs 6 times.
\item [{\rm -}] $\rm{S}_4\cap \rm{S}_4\cong \rm{S}_3.$ Occurs 3 times.
\item [{\rm -}] $\rm{S}_4\cap C_{13} \rtimes C_3\cong C_3.$ Occurs 18 times.
\item [{\rm -}] $(C_3^2\rtimes Q_8)\rtimes C_3\cap (C_3^2\rtimes Q_8)\rtimes C_3\cong (C_3^2\rtimes C_3)\rtimes (C_2^2).$ Occurs 1 times.
\item [{\rm -}] $(C_3^2\rtimes Q_8)\rtimes C_3\cap C_{13} \rtimes C_3\cong C_3.$ Occurs 12 times.
\item [{\rm -}] $C_{13} \rtimes C_3\cap C_{13} \rtimes C_3\cong C_3.$ Occurs 15 times.

    \end{itemize}

    \noindent \textbf{Case 3:} $|x|=3, |N_G(\gen{x})|=108.$
    \begin{itemize}
        \item [{\rm -}] $(C_3^2\rtimes Q_8)\rtimes C_3\cap (C_3^2\rtimes Q_8)\rtimes C_3\cong GL(2,3).$ Occurs 9 times.
\item [{\rm -}] $(C_3^2\rtimes Q_8)\rtimes C_3\cap (C_3^2\rtimes Q_8)\rtimes C_3\cong (C_3^2\rtimes C_3)\rtimes (C_2^2).$ Occurs 7 times.
\item [{\rm -}] $(C_3^2\rtimes Q_8)\rtimes C_3\cap (C_3^2\rtimes Q_8)\rtimes C_3\cong \rm{S}_3 \times \rm{S}_3.$ Occurs 12 times.
    \end{itemize}
    
    \noindent\textbf{Case 4:} $|x|=4$
\begin{itemize}
        \item [{\rm -}] $(C_3^2\rtimes Q_8)\rtimes C_3\cap (C_3^2\rtimes Q_8)\rtimes C_3\cong GL(2,3).$ Occurs 1 time.
\item [{\rm -}] $(C_3^2\rtimes Q_8)\rtimes C_3\cap \rm{S}_4\cong D_8.$ Occurs 4 times.
\item [{\rm -}] $\rm{S}_4\cap \rm{S}_4\cong D_8.$ Occurs 1 time.
    \end{itemize}

    \noindent\textbf{Case 5:} $|x|=6.$
    \begin{itemize}
        \item [{\rm -}] $(C_3^2\rtimes Q_8)\rtimes C_3\cap (C_3^2\rtimes Q_8)\rtimes C_3\cong (C_3^2\rtimes C_3)\rtimes (C_2^2).$ Occurs 3 times.
\item [{\rm -}] $(C_3^2\rtimes Q_8)\rtimes C_3\cap (C_3^2\rtimes Q_8)\rtimes C_3\cong \rm{S}_3 \times \rm{S}_3.$ Occurs 2 times.
\item [{\rm -}] $(C_3^2\rtimes Q_8)\rtimes C_3\cap (C_3^2\rtimes Q_8)\rtimes C_3\cong GL(2,3).$ Occurs 1 time.
    \end{itemize}

    \noindent\textbf{Case 6:} $|x|=8.$
    \begin{itemize}
        \item [{\rm -}] $(C_3^2\rtimes Q_8)\rtimes C_3\cap (C_3^2\rtimes Q_8)\rtimes C_3\cong GL(2,3).$ Occurs 1 time.
    \end{itemize}

    \noindent\textbf{Case 7:} $|x|=13.$
    \begin{itemize}
        \item [{\rm -}] No intersections.
    \end{itemize}    


\begin{thebibliography}{99} 

\bibitem{Abe} S. Abe, A characterization of some finite simple groups by orders of their soluble subgroups, \textit{Hokkaido Math. J.}, \textbf{31}, (2002) 349-361.


\bibitem{Akbari2} B. Akbari, M.\,L. Lewis, J. Mirzajani and A. R. Moghaddamfar, The solubility graph associated with a finite group, {\it Internat. J. Algebra Comput.} {30}(8) (2020) 1555-1564.

\bibitem{Akbari3} B. Akbari, C. Delizia, and C. Monetta, On the solubilizer of an element in a finite group. \textit{Mediterr. J. Math.}, \textbf{20}, 135 (2023).

\bibitem{Julia} J. Bezanson, A. Edelman, S. Karpinski, and V. B. Shah, Julia: A Fresh Approach to Numerical Computing. \textit{SIAM}, \textbf{59} (1), (2017).

\bibitem{Bray} J. N. Bray, D. F. Holt, and C. M. Roney-Dougal, \textit{The maximal subgroups of the low-dimensional finite classical groups}, Cambridge University Press, (2013).


\bibitem{Bhowal} P. Bhowal, D. Nongsiang, and R. K. Nath, Solvable graphs of finite groups. \textit{Hacet. J. Math. Stat.}, \textbf{49} (6), (2020) 1955-1964.

\bibitem{Burness} T. C. Burness, A. Lucchini, and D. Nemmi, On the soluble graph of a finite group, \textit {J. Comb.}, \textbf{194} (2023).

\bibitem{Atlas} J. H. Conway, R. T. Curtis, S. P. Norton, R. A. Parker and R. A. Wilson, {\em Atlas of
Finite Groups} Clarendon Press, Oxford University Press, Eynsham, (1985).

\bibitem{Dummit} D. S. Dummit and R. M. Foote, \textit{Abstract Algebra. 3rd Edition}, John Wiley \& Sons, Inc, Hoboken, (2004).

\bibitem{Flavell} P. Flavell, Finite groups in which every two elements generate a soluble group, \textit{Invent. Math.}, \textbf{121},  (1995) 279-285.

\bibitem{GAP} The GAP Group, GAP - Groups, Algorithms, and Programming, Version 4.12.2 (2022). https://www.gap-system.org.

\bibitem{GKPS} R. Guralnick, B. Kunyavski${\rm \breve{i}}$, E. Plotkin and A. Shalev,  Thompson-like characterization of the solvable radical, \textit{J. Algebra}, \textbf{300} (1), (2006) 363-375.

\bibitem{GW} R. Guralnick and  J. S. Wilson, The probability of generating a finite soluble group, \textit{Proc. Lond. Math. Soc.}, \textbf{81} (2), (2000), 405-427.

\bibitem{HaiReuven} D. Hair-Reuven, Non-solvable graph of a finite group and solvabilizers. \textit{ArXiv Preprint}, arxiv:1307.2924v1.

\bibitem{Huppert} B. Huppert, {\em Endliche Gruppen I}, Springer, Berlin, 1967.

\bibitem{King} O. H. King, The subgroup structure of finite classical groups in terms of geometric configurations. From \textit{Surveys in Combinatorics}, Cambridge University Press, Cambridge, (2005) 29-56.

\bibitem{Mousavi} H. Mousavi, M. Poozesh, and Y. Zamani, The impact of the solubilizer of an element on the structure of a finite group. \textit{ Ricerche mat.} (2023).


\bibitem{Suzuki} M. Suzuki, Finite groups of even order in which Sylow $2$-groups are independent. \textit{Ann. Math.}, \textbf{80}, (1964), 58-77.

\bibitem{Thompson} J. H. Thompson, Nonsolvable finite groups all of whose local subgroups are solvable. \textit{Bull. Am. Math. Soc.}, \textbf{48} (2), (1973), 511-592.

\bibitem{Wujie} W. J. Shi, A characterization of Suzuki's simple groups, \textit{PROC}, \textbf{114} (3), (1992) 589-591.

\end{thebibliography}
\end{document}